\documentclass[11pt]{article}
\usepackage{amsfonts}
\usepackage{amssymb}
\usepackage{amsthm}
\usepackage{amsmath}
\usepackage{braket}
\usepackage{graphicx}
\usepackage{indentfirst}
\usepackage{cite}
\usepackage{mathrsfs}
\usepackage{graphics}
\usepackage{tikz}
\usetikzlibrary{decorations.pathreplacing,angles,quotes}
\usepackage{subcaption}
\usepackage{xcolor}
\newtheorem{thrm}{\textbf{Theorem}}[section]
\newtheorem{lmm}{\textbf{Lemma}}[section]
\newtheorem{prpstn}{\textbf{Proposition}}[section]
\newtheorem{corollary}{\textbf{Corollary}}[section]
\newtheorem{rmrk}{\textbf{Remark}}[section]
\newtheorem{dfntn}{\textbf{Definition}}[section]

\allowdisplaybreaks[4]

\usepackage{chngcntr}
\counterwithin*{equation}{section}

\def\be{\begin{equation}}
\def\ee{\end{equation}}
\def\bea{\begin{eqnarray}}
\def\eea{\end{eqnarray}}
\def\bt{\begin{theorem}}
\def\et{\end{theorem}}
\def\bl{\begin{lemma}}
\def\el{\end{lemma}}
\def\br{\begin{remark}}
\def\er{\end{remark}}
\def\bp{\begin{proposition}}
\def\ep{\end{proposition}}
\def\bc{\begin{corollary}}
\def\ec{\end{corollary}}
\def\bd{\begin{definition}}
\def\ed{\end{definition}}

\def\R3{\mathbb{R}^3} 
\def\R{\mathbb{R}}
\def\F2o{\overline{F_2}}

\def \E{\mathcal E}

\def \eps{\varepsilon}
\newcommand{\ej}[0]{\eps_j}
\newcommand{\uj}[0]{u_j}
\newcommand{\ux}[0]{u_{j,x}}

\newcommand{\uint}[0]{\int_0^1}
\newcommand{\ai}[0]{\alpha_i}
\newcommand{\bi}[0]{\beta_i}

\newcommand{\beq}[0]{\begin{equation}}
\newcommand{\eeq}[0]{\end{equation}}

\newcommand{\Mint}[0]{\int_0^1}

\newcommand{\EE}[0]{{\mathcal E}}

\DeclareMathOperator{\supp}{supp}
\DeclareMathOperator{\argmin}{argmin}

\begin{document}

\title{Interfacial energy as a selection mechanism for minimizing gradient Young measures in a one-dimensional model problem}

\author{
{\sc Francesco Della Porta}\footnote{Mathematical Institute, University of Oxford, Oxford OX2 6GG, UK \textit{dellaporta@maths.ox.ac.uk}}
}
\date{\today}
\maketitle

%
%
\begin{abstract}
Energy functionals describing phase transitions in crystalline solids are often non-quasiconvex and minimizers might therefore not exist. On the other hand, there might be infinitely many gradient Young measures, modelling microstructures, generated by minimizing sequences, and it is an open problem how to select the physical ones. \\
In this work we consider the problem of selecting minimizing sequences for a one-dimensional three-well problem $\EE$. We introduce a regularization $\E^\eps$ of $\E$ with an $\eps$-small penalization of the second derivatives, {and we obtain as $\eps\downarrow0$ its $\Gamma-$limit and, under some further assumptions, the $\Gamma-$limit of a suitably rescaled version of $\E^\eps$.} The latter selects a unique minimizing gradient Young measure of the former, which is supported just in two wells and not in three. {We then show that some assumptions are necessary to derive the $\Gamma-$limit of the rescaled functional,} but not to prove that minimizers of $\E^\eps$ generate, as $\eps\downarrow 0$, Young measures supported just in two wells and not in three.
\end{abstract}
%
%
%
\maketitle
\section{Introduction}
A common problem that arises when studying martensitic transformations in the context of nonlinear elasticity (see e.g., \cite{BallJames1,BallJames2,Batt,MullerDispense}) is to minimize an energy functional
$$
E(y) = \int_\Omega \phi (\nabla y(x))\,\mathrm dx,
$$ 
where $\Omega$ is an open and bounded Lipschitz domain, and $y\colon \Omega\to\R^3$ is a map in a suitable Sobolev space satisfying $y=\bar y$ on $\partial\Omega$, for some smooth enough mapping $\bar{y}$. In this context, the continuous function $\phi\colon \R^{3\times 3} \to [0,+\infty]$ is generally such that
$$
\phi(F) = 0\qquad \Longleftrightarrow \qquad F\in \mathcal K:=\sum_{i=1}^n SO(3)U_i,
$$ 
where $n\geq 1$ and $U_i$ are positive definite symmetric matrices representing the different variants of martensite. As in general $E$ is not quasiconvex, minimizers for this energy might not exist. Therefore, following the idea of \cite{BallJames1} one can study the behaviour of minimizing sequences, having a gradient that tends in measure to $\mathcal{K}$, and characterised by interesting microstructures. {In order to capture the limiting behaviour of the minimising sequences, one can study the relaxed functional
$$
\bar E(\nu) = \int_\Omega \int_{\R^{3\times3}} \phi (F)\,\mathrm d\nu_x(F)\,\mathrm dx,
$$
where $\nu$ is a gradient Young measure containing the information {about} microstructures in the crystal (see e.g., \cite{BallJames2,MullerDispense,Pedregal}).} 
Defining $\mathcal{M}_1(\R^{3\times3})$ as the set of probability measures on $\R^{3\times3},$ let us consider
\[
\mathcal A:= \Set{\nu \in L^\infty_{w^*}(\Omega;\mathcal{M}_1(\R^{3\times3}))\,\bigg|\; \text{\parbox{3.5in}{\centering $\supp\nu_x\subset \mathcal K$, $\exists y\in W^{1,\infty}(\Omega;\R^3)$ s.t. $y=\bar{y}$ on $\partial\Omega,$ and $\int_{\R^{3\times3}} F\mathrm d\nu_x(F) = \nabla y(x)$  a.e. in $\Omega$}}},
\]
and notice that this set is the set of minimizers of $\bar E$ whenever $\min\bar E=0$. Here, we denoted by $L^\infty_{w^*}(\Omega;\mathcal M_1(\R^{3\times 3}))$ the space $L^\infty(\Omega;\mathcal M_1(\R^{3\times 3}))$ endowed with the weak$*$ topology. The solutions constructed in \cite{MullerSverak} with the technique of convex integration, show that the set $\mathcal A$ might contain infinitely many minimizers for $\bar E$, and its elements might sometimes appear non-physical. In agreement with the physics, many authors in the literature (see e.g., \cite{BallCrooks,BallJames1,MullerDolzman, Muller1,Leoni,ContiSchweizer}) have considered a regularization of $E$ that penalizes the second derivatives of $y$ such as
\beq
\label{E originale}
E^\eps(y) = \int_\Omega \bigl(\eps^2|\nabla^2y|^2+\phi (\nabla y(x))\bigr)\,\mathrm dx,\qquad\text{ or }\qquad \tilde E^\eps(y) = \eps|\nabla^2y|(\Omega)+\int_\Omega \phi (\nabla y(x))\,\mathrm dx.
\eeq
Here, $\eps>0$ is small and $|\nabla^2y|(\Omega)$ is the norm of $\nabla^2y$ as a measure on $\Omega.$ Many results have been proved in the case $n=2$ and without boundary conditions. For example, it is proved in \cite{MullerDolzman} that the requirement $\nabla y \in BV(\Omega;\mathcal K)$ forces the gradient discontinuities to be just on planes that never intersect in $\Omega$.  
In \cite{Leoni} the limit solutions for $E^\eps$ as $\eps\to0$ when $\mathcal K=\{A,B\}$ are characterized via a $\Gamma$-limit argument. In the two-dimensional setting with $\mathcal K=\{SO(2)A,SO(2)B\}$ the generalised $\Gamma$-limit has been analysed in \cite{ContiSchweizer}, and strongly exploits the above mentioned result of \cite{MullerDolzman}.\\ 

More generally, we could argue that the physically relevant minimizers of $\bar E$ are not those in $\mathcal A$, but those belonging to the subset
$$
\mathcal B:=  \Set{ \nu\in\mathcal A \,\bigg|\; \text{\parbox{3.5in}{\centering $\exists$ 
minimizers $u^{\eps_j}$ of $E^{\eps_j}$, with $\eps_j\downarrow0$, such that $\delta_{\nabla u^{\eps_j}}\to \nu$ in $L^\infty_{w^*}(\Omega;\mathcal M_1(\R^{3\times 3}))$}
}},
$$
or equivalently $\tilde{\mathcal B}$ where $E^{\eps_j}$ is replaced by $\tilde E^{\eps_j}$. \\

Finding an explicit characterization for $\mathcal B$ seems however out of reach for the general three-dimensional problem. For this reason, in this work we focus on the one-dimensional energy functional
\beq
\label{energia}
\EE(u)  = \Mint \bigr( W(u_x)+u^2\bigl)\,\mathrm dx,
\eeq
which has been often considered in the literature (see e.g., \cite{BallPego,Muller1,Muller2,Nicolaides}) as a one-dimensional prototype for $E$. {
Indeed, the role of the boundary conditions in more dimensions is played here by the term $u^2$ in the energy, which forces the $L^2-$norm of the minimisers (or of the minimising sequences) to be close to a prescribed value, which is chosen to be null for simplicity, and whose gradient does not sit on the wells. }
Suppose $W$ satisfies
\begin{itemize}
\item[(H1)] $W:\R\mapsto \mathbb{R}_+$ is a continuous non-negative function;
\item[(H2)] there exist $c_1,c_2,c_3>0$ and $p\in (1,\infty)$ such that
$$
c_1|s|^p - c_2 \leq W(s)\leq c_3 (|s|^p+1),\qquad \forall s \in \R;
$$
\item[(H3)] $W(s)=0$, for each $s\in \mathcal Z$, and $W(s)>0,$ otherwise,
where $$\mathcal Z:=\{s\in \R:\,s\in \argmin(W)\}. $$
\end{itemize}
\begin{figure}
  \centering
  \includegraphics[width=.59\linewidth]{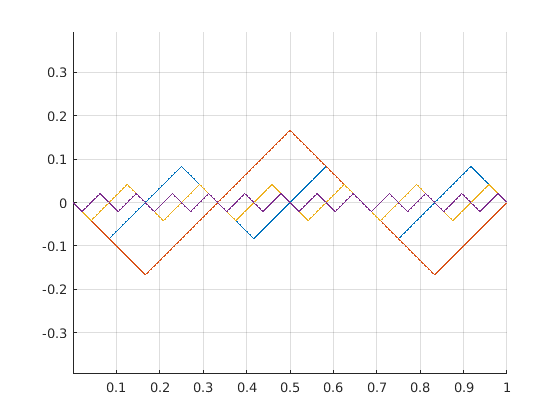}
  \caption{Minimizing sequences for $\EE$ when $0\notin\mathcal Z$.}
  \label{sawtooth}
\end{figure}
If $\mathcal Z$ has a finite number of elements, if there exist $z_1,z_2\in\mathcal Z$ with $z_1<0<z_2$, and if $0\notin \mathcal Z,$ then $W$ is not convex and $\EE$ does not have minimizers in $W^{1,p}_0(0,1)$. Indeed, by constructing arbitrarily small saw-tooth functions (cf. Figure \ref{sawtooth}) with gradient in $\mathcal{Z}$ one can show that $\inf \mathcal{E} = 0$. Therefore, the existence of a minimizer $u\in W^{1,p}_0(0,1)$ would imply $u=0$, and hence $u_x=0$ a.e. in $(0,1)$, which is in contradiction with the fact that, by (H3), $W(0)\neq 0$. For this reason, we consider the regularized problem 
\beq
\label{energia2}
\EE^\eps(u)  = 
\begin{cases}
\Mint\bigl(\eps^6|u_{xx}|^2 + W(u_x)+u^2\bigr)\,\mathrm{d}x, &\qquad\text{if }u\in W^{1,p}_0(0,1)\cap H^2(0,1),\\
+\infty, &\qquad\text{otherwise},
\end{cases}
\eeq
which is the one-dimensional analogue of \eqref{E originale}. $\EE^\eps$ can also be rewritten by using gradient Young measures (see e.g., \cite{MullerDispense,Pedregal}) as
\beq
\label{energia3}
\bar\EE^\eps(u,\nu)  = 
\begin{cases}
\EE^\eps(u),&\qquad \text{if $u\in W^{1,p}_0(0,1)\cap H^2(0,1)$ and $\nu_x=\delta_{u_x(x)}$ a.e. in $(0,1)$},\\
+\infty, &\qquad\text{otherwise},
\end{cases}
\eeq
with $\delta_s$ denoting the Dirac mass at $s$.
In this case the problem admits a solution in $W^{1,p}_0(0,1)\cap H^2(0,1)$ and the question arises as to what happens to the limit of the minimizers $u^\eps$ as $\eps\downarrow0$. 
For every $u\in W^{1,p}_0(0,1)$ let us define the set of its gradient Young measures
\begin{align*}
&\mathrm{GYM}^p(u):=
\Set{
\nu\in L_{w^*}^\infty(0,1;\mathcal{M}_1(\R))
 \bigg|\; \text{\parbox{3.0in}{\centering 
 $\int_\R s\,\mathrm d\nu_x(s) = u_x(x)$ a.e. 
 $x\in(0,1)$,  $\Mint \int_\R \lvert s\rvert^p \, \mathrm d\nu_x(s)\,\mathrm dx<\infty$
 }}},\\
&\mathrm{GYM}^\infty(u):=
\Set{\nu\in L_{w^*}^\infty(0,1;\mathcal{M}_1(\R))
 \bigg|\; \text{\parbox{3.0in}{\centering $\int_\R s\,\mathrm d\nu_x(s) = u_x,\, \supp \nu_x\subset K$ a.e. in $(0,1)$, $K\subset\R$ compact}
}}.
\end{align*}
Here, $\mathcal M(\R)$ and $\mathcal M_1(\R)$, often abbreviated below by $\mathcal M$ and $\mathcal M_1$, are respectively the space of bounded Radon measures $\mu$ on $\R$, and its subset of probability measures. A preliminary result that is proved later in Section \ref{section 2} is the following 
\begin{prpstn}
\label{GL general}
Let $W$ satisfy {\rm(H1)--(H3)}. Then, $\bar\EE^\eps$ $\Gamma-$converges to 
\[\EE^0(u,\nu)=
\begin{cases}
\Mint \bigl(\langle \nu_x,W \rangle +u^2\bigr) \mathrm{d}x,&\qquad\text{if }u\in W^{1,p}_0(0,1), \nu \in \mathrm{GYM}^{p}(u),\\
+\infty, &\qquad\text{otherwise},
\end{cases}
\]
in the $L^2(0,1)\times L^\infty_{w^*}(0,1;\mathcal{M})$ topology as $\eps$ tends to $0$.
\end{prpstn}
If $\mathcal{Z}=\{z_1,z_2\}$ with $z_1<0<z_2$, then under (H1)--(H3) minimizers $(u,\nu)$ of $\EE^0$ must satisfy
\beq
\label{minimi}
u(x)=0,\qquad \nu\in \mathrm{GYM}^p(0),\qquad \supp\nu_x\in\mathcal Z,\qquad\text{a.e. in $(0,1)$}.
\eeq
These conditions determine a unique minimizer $(u,\nu)$ to $\EE^0$, namely
$$
u(x)=0,\qquad \nu_x = \frac{z_2}{z_2-z_1}\delta_{z_1}  -\frac{z_1}{z_2-z_1}\delta_{z_2},\qquad\text{a.e. $x$ in $(0,1)$}.
$$
Let us assume
\begin{itemize}
\item[(H4)] $\mathcal{Z} = \bigl\{z_1,z_2,z_3\bigr\}$, and, without loss of generality, that $z_1<0<z_2< z_3.$
\end{itemize}
{In this case, given any arbitrary measurable 
$$\lambda\colon (0,1)\to\Bigl[0,\Bigl(1-\frac{z_2}{z_1}\Bigr)^{-1}\Bigr],$$
the pair $(u,\nu)$ defined for almost every $x\in(0,1)$ by $u(x) = 0$ and
\beq
\label{1puntoq}
\nu_x = -\frac{z_3+\lambda(x)(z_2-z_3)}{z_1-z_3}\delta_{z_1} + \lambda(x)\delta_{z_2}+\frac{z_1+\lambda(x)(z_2-z_1)}{z_1-z_3}\delta_{z_3},
\eeq
minimises $\mathcal{E}^0.$ 
}
As a consequence, by assuming (H1)--(H4), uniqueness of minimizers for $\EE^0$ is lost, that is, the gradient of the minimizing sequences for $\mathcal{E}$ oscillate, and converge in measure to $\{z_1,z_2,z_3\}$ without any particular preference. The aim of this work is to prove that minimizers of $\EE^\eps$ generate gradient Young measures supported in $\{z_1,z_2\}$, but not in $z_3$. Therefore, by choosing minimisers of $\EE^\eps$ with $\eps\downarrow0$ as minimizing sequences for $\EE$ we can select a unique minimising gradient Young Measure, out of the infinitely many given above.\\

Let $$V:=H^2(0,1)\cap W^{1,p}_0(0,1).$$ Then we define $I^\eps$ by
$$
I^\eps (u) = I^\eps(u,\nu):=
\begin{cases}
\eps^{-2}\int_0^1\bigl(\eps^6u_{xx}^2+W(u_x)+u^2	\bigr)\,\mathrm dx,&\qquad\text{if }u\in V,\, \nu_x=\delta_{u_x(x)},\\
+\infty, &\qquad\text{otherwise}.
\end{cases}
$$
We remark that this problem was thoroughly studied in \cite{Muller1,Muller2}, under the assumption that $W$ is a double-well potential, and where quasi-periodicity of the minimizers was also proved. As shown below, however, generalization to a three well problem is non-trivial and requires a good understanding on the possible shape of the minimizing sequences. We also point out that the behaviour of $I^\eps$ is different from the one of Modica-Mortola type functionals (see e.g., \cite{CSZ,MM}) as $\eps\downarrow 0$. Indeed, in our case the term in $u^2$ forces minimizers of $I^\eps$ to oscillate faster and faster as $\eps\downarrow 0$, making the number of oscillations in the gradient tend to infinity. In what follows we define
$$E_0:=2\int_{z_1}^{z_2}|W(s)|^\frac{1}{2}\mathrm{d}s,\qquad E_1:=2\int_{z_2}^{z_3}|W(s)|^\frac{1}{2}\mathrm{d}s,$$
and 
\begin{align*}
A_0 &:= \inf_d\bigl(3^{-1}z_2^2z_{21}d^2+E_0d^{-1}) = (2^{-1}3)^{\frac23}E_0^{\frac23}\bigl(z_2^2 z_{21}\bigr)^{\frac13},\\
B_0 &:=\inf_d\bigl(3^{-1}z_3^2 z_{31}d^2+(E_0+E_1)d^{-1}) = (2^{-1}3)^{\frac23}(E_0+E_1)^{\frac23}\bigl(z_3^2z_{31}\bigr)^{\frac13}.
\end{align*}
where $z_{i1}:=(1-\frac{z_i}{z_1})$ for $i=2,3.$ Further assumptions on $W$ are: 
\begin{itemize}

\item[(H5)](Coercivity) There exist $\eta_0\in(0,\min\{1,-z_1,z_2,\frac{z_3-z_2}2\})$, $c_0>0$, $q>0$ 
such that 
$$
W(s)\geq c_0 \min\bigl\{\min_i|s-z_i|^q,\,|\eta_0|^q\bigr\}, \qquad\forall s\in\R;
$$
\item[(H6)] Let $f_6(y):=9(E_0+E_1)^2\bigl(z_2^2+y^3z_3^2+3yz_2(yz_3+z_2)\bigr),$ then 
$$
f_6(y)-\bigl(A_0+B_0 y \bigr)^3\geq 0,\quad\text{for every $y\geq0; $}
$$
\item[(H7)] Let $f_7(y):=\frac94(E_0+2E_1)^2\bigl(z_2^2z_{21}+y^3z_3^2z_{31}+3yz_2z_{31}(yz_3+z_2)\bigr),$ then 
$$
f_7(y)-\bigl(A_0+B_0 y \bigr)^3\geq 0,\quad\text{for every $y\geq0; $}
$$
\item[(H8)] Let
\[
f_8(y):=9(E_0 + E_1)^2\biggl(z_2^2z_{21}+y^3z_3^2z_{31}-3
\frac{(y^2 z_{31}z_3-z_2z_{21})^2}{4(z_{21}+yz_{31})}\biggr),
\]
then,
 $$
 f_8(y)-\bigl(A_0+B_0 y \bigr)^3\geq 0,\quad\text{for every $y\geq0; $}
 $$
\end{itemize}
\noindent
These technical assumptions are used to guarantee that the microstructures constructed in Section \ref{section 3} are energetically preferable to those constructed in Proposition \ref{H7 prop} and Proposition \ref{H8 prop} (see also Figure \ref{fig:fig}). Here by microstructure we mean the shape of a building block which is repeated quasi-periodically in configurations of low energy for $I^\eps$. The period gets smaller with $\eps$. The preferred microstructure clearly depends on the position of the wells, that is on $z_1,z_2,z_3$, and on the cost of passing from one well to the other, that is on $E_0,E_1$. (H6) and (H7) reduce to checking that two cubic polynomials are non-negative on $\R_+.$ (H6)--(H8) can be verified easily with a computer and hold in a wide range of cases. We refer the reader to Section \ref{section 7} for more details and for a couple of examples.\\

The first result that we prove is a second $\Gamma-$limit for $\bar {\mathcal \EE}$, {that is a $\Gamma-$limit result for $I^\eps$}
\begin{thrm}
\label{main Thm}
Assume {\rm(H1)-(H8)}. Then $I^\eps(u,\nu)$ $\Gamma-$converges in the $L^2(0,1)\times L^\infty_{w*}(0,1;\mathcal M)$ topology to
$$
I^0(u,\nu) = 
\begin{cases}
A_0\Mint \nu_x(z_2)\,\mathrm dx+B_0 \Mint \nu_x(z_3)\,\mathrm dx ,\qquad &\text{if $u=0$, $\nu \in \mathrm{GYM}^\infty(0),$ $\supp\nu_x\subset \mathcal{Z}$ a.e. ,}\\ 
+\infty,\qquad &\text{otherwise}.
\end{cases}
$$
\end{thrm}
We remark that, as $\nu\in \mathrm{GYM}^\infty(0),$ and $\supp\nu\subset \mathcal{Z}$ a.e., we must have 
$$\Mint \nu_x(z_3)\,\mathrm dx = z_{31}^{-1}-\frac{z_{21}}{z_{31}}\Mint \nu_x(z_2)\,\mathrm dx.$$ On the other hand $A_0<\frac{z_{21}}{z_{31}}B_0$, so that $I^0(0,\nu)$ is a linearly decreasing function of $\Mint \nu_x(z_2)$. Therefore, the minimum {of $I^0$ is attained at}
$$\Mint \nu_x(z_3)\,\mathrm dx=0,\qquad \Mint \nu_x(z_2)\,\mathrm dx=z_{21}^{-1}.$$ 
Thus minimizing sequences for $\mathcal E^\eps$ have gradients tending in measure to $\{z_1,z_2\}$, and $z_3$ is not seen in the limit. That is, the vanishing interfacial energy limit selects a unique minimizer out of the infinitely many minimizers of $\EE^0$.\\

As shown in Section \ref{remarks on ass}, (H7) and (H8) are necessary conditions to prove the above $\Gamma-$limit result. Nonetheless, it turns out that we can characterize the set of gradient Young measures generated by minimizing sequences for $I^\eps$, even without the second $\Gamma-$limit for $\bar {\mathcal \EE}$. This is the result of the following theorem, where also (H6) is relaxed:
{\begin{thrm}
\label{main Thm 2}
Assume {\rm(H1)--(H5)} and $z_3\leq 3|z_1|$. Then any sequence $u^{j}\in V$ of minimizers for $\E^{\eps_j}$, with $\eps_j\to0$, is such that $u^{j}\to 0$ in $L^2(0,1)$, $\delta_{u_x^{j}}\to \nu$ in $L^\infty_{w*}(0,1;\mathcal M)$, and $\nu\in \mathrm{GYM}^\infty(0)$ satisfies  
$$\supp\nu_x\in \{z_1,z_2\},\qquad\nu_x = \frac{z_2}{z_2-z_1}\delta_{z_1}-\frac{z_1}{z_2-z_1}\delta_{z_2},\qquad \text{a.e. $x\in(0,1)$.}  $$
\end{thrm}
}
{In this way we have shown that, in our case, even if the set of gradient Young measures minimizing $\EE^0$ has infinitely many elements, its subset generated by minimizers for $\E^\eps$, which are also minimizers for the regularized and rescaled problem $I^\eps$, contains just one element.}\\
Therefore, the one-dimensional model problem studied in this paper confirms that vanishing interface energy can be used as a tool to select minimizing gradient Young measures. This suggests that for the three-dimensional problem $E$ the set $\mathcal{B}$ is actually much smaller than $\mathcal A$. Furthermore, our results show that the shape of the second $\Gamma-$limit for $E^\eps$ might change with the shape of $\phi$. Nonetheless, as in our model problem, it might be possible to characterize $\mathcal B$ independently of the second $\Gamma-$limit for $E^\eps$.\\

The plan for the paper is the following: in Section \ref{section 2} we prove Proposition \ref{GL general}, in Section \ref{section 3} and \ref{section 4} we compute some upper and lower bounds for $I^\eps$. Section \ref{section 5} is devoted to prove Theorem \ref{main Thm}, while Section \ref{Che bound} is devoted to prove Theorem \ref{main Thm 2}. Finally, in Section \ref{remarks on ass} we sketch necessity of (H7)--(H8) and give an example where (H7)--(H8) hold, and one where they don't.\\

In the following sections we will denote by $c$ a generic positive constant depending only on the parameters of the problem, and not on the quantities $N,M,N_\eps,M_\eps,\eta,\eps,\mu,j,\sigma$ appearing below. Its value may change from line to line or even within the same line.

\section{Proof of the first $\Gamma-$limit}
\label{section 2}
In this section we prove Proposition \ref{GL general}.\\

We first observe that, as $\bar \EE^\eps (u,\nu)$ is a monotone sequence in $\eps$, the $\Gamma-$limit exists and is given by the lower semicontinuous envelope of the pointwise limit of the sequence (cf. \cite[Remark 1.40]{Braides}). That is, the $\Gamma-$limit is given by
\beq
\label{G1}
sc\begin{cases}
\Mint \bigl(\langle \nu_x,W \rangle +u^2\bigr) \mathrm{d}x,&\qquad\text{if }u\in V,\nu_x=\delta_{u_x(x)} \text{ a.e. in $(0,1)$},\\
+\infty, &\qquad\text{otherwise},
\end{cases}
\eeq
where $sc$ denotes the lower semicontinuous envelope with respect to the topology $L^2(0,1)\times L^\infty_{w^*}(0,1;\mathcal{M})$. We first claim that \eqref{G1} is equal to
\beq
\label{G1 bis}
sc\begin{cases}
\Mint \bigl(\langle \nu_x,W \rangle +u^2\bigr) \mathrm{d}x,&\qquad\text{if }u\in W^{1,p}(0,1),\nu_x\in \mathrm{GYM}^p(u),\\
+\infty, &\qquad\text{otherwise},
\end{cases}
\eeq
that is we can relax the requirements $u\in V,\nu_x=\delta_{u_x(x)}$ a.e. in $(0,1)$. Indeed, given an $u\in W^{1,p}_0(0,1)$, we can approximate it by $u^j\in H^2(0,1)$ such that $u^j\to u$ strongly in $W^{1,p}_0(0,1)$. Therefore, by passing into the limit as $j$ tends to $\infty$ we can drop the requirement $u\in H^2(0,1)$ in \eqref{G1}. Now, let $\nu\in \mathrm{GYM}^p(u)$ for some $u\in W^{1,p}_0(0,1)$. Then by \cite[Thm. 8.7]{Pedregal} we know the existence of a sequence $u^j\in W^{1,p}(0,1)$ converging weakly to $u$ in $W^{1,p}(0,1)$, strongly in $L^2(0,1)$, such that $\delta_{u^j_x}$ converges to $\nu$ in $L^\infty_{w^*}(0,1;\mathcal{M})$. Thanks to \cite[Lemma 8.3]{Pedregal} the sequence can actually be chosen in $W^{1,p}_0(0,1)$. Therefore, the fact that (cf. \cite[Thm. 6.11]{Pedregal})
$$
\liminf_j \Mint \langle\delta_{u^j_x},W \rangle\,  \mathrm{d}x \geq\Mint \langle\nu_x,W \rangle \, \mathrm{d}x,
$$
allows us to drop also the requirement on $\nu$ that $\nu_x=\delta_{u_x(x)}$ for a.e. $x\in(0,1)$, concluding the proof that \eqref{G1} is equal to \eqref{G1 bis}. We now claim that we can drop $sc$ from \eqref{G1 bis}, that means, that 
\beq
\label{G1 tris}
\begin{cases}
\Mint \bigl(\langle \nu_x,W \rangle +u^2\bigr) \mathrm{d}x,&\qquad\text{if }u\in W^{1,p}(0,1),\nu_x\in \mathrm{GYM}^p(u),\\
+\infty, &\qquad\text{otherwise},
\end{cases}
\eeq
is already lower semicontinuous in the $L^2(0,1)\times L^\infty_{w^*}(0,1;\mathcal{M})$ topology.
To prove this claim, it is sufficient to show that for every sequence $(u_j,\nu^j)\in L^2(0,1)\times \mathrm{GYM}(u_j)$ converging to $(u,\nu)$ in $L^2(0,1)\times L^\infty_{w^*}(0,1;\mathcal{M})$, we have $\liminf_j \EE^0(u_j,\nu^j) \geq\EE^0(u,\nu)$. We will follow the approach devised in \cite{BallKoumatos}. If $\liminf_j\EE^0(u_j,\nu^j)=\infty$, the thesis follows trivially. Therefore, by passing without loss of generality to a subsequence, we can assume $\EE^0(u_j,\nu^j)\leq C$. By (H2), this implies that
$$
\Mint \langle \nu^j_x,|\cdot|^p\rangle \,\mathrm dx\leq C,
$$
and, by \cite[Thm. 3.6]{Sychev}, we deduce that $\nu_x$ is a probability measure for almost every $x\in(0,1)$. Jensen's inequality and the fact that $|\cdot|^p$ is convex yield
$$
\Mint |\bar\nu^j_x|^p\,\mathrm dx\leq \Mint \langle \nu^j_x,|\cdot|^p\rangle\,\mathrm dx \leq C,\qquad \text{where}\qquad \bar\nu^j := \int_\R s\,\mathrm d\nu^j(s) .
$$
It follows that $\bar\nu^j\rightharpoonup \bar\nu$ in $L^p(0,1)$ and, therefore, that $u_j\rightharpoonup u$ in $W^{1,p}(0,1)$, where $u_x = \bar\nu$. A result like the one in \cite[Prop. 4.5]{Sychev} finally gives us that $\nu\in \mathrm{GYM}^p(u)$. At this point, an application of \cite[Prop. 3.7]{Sychev} allows us to deduce that $\liminf_j \EE^0(u_j,\nu^j) \geq\EE^0(u,\nu)$, thus concluding the proof.

\begin{rmrk}
{\rm Following the same strategy it is actually possible to prove that $E^\eps$ and $\tilde{E}^\eps$ $\Gamma-$converge in the $L^1(\Omega)\times L^\infty_{w^*}(\Omega;\mathcal M_1(\R^{3\times3}))$ topology to $\bar E$ as $\eps\to0$.}
\end{rmrk}

\section{Construction of an upper bound}
\label{section 3}

In this section we prove the following proposition:
\begin{prpstn}
\label{upper bnd}
{Assume {\rm(H1)--(H5)}, let $n\in\mathbb{N}$, $n\geq 2$, and let $0=x_1< x_2 < \dots < x_n = 1$ be a partition of $[0,1]$. There exist $\zeta>0$ and $\eps_0=\eps_0(\min_i (x_{i+1}-x_i))>0$, such that for every 
$\eps\leq\eps_0$ we can find $u\in V$ with
\beq
\label{sup goal}
\begin{split}
\int_{x_i}^{x_{i+1}} \bigl(\eps^4 u^2_{xx}+ \eps^{-2}W(u_x)+ \eps^{-2}u^2\bigr)\,\mathrm dx\leq A_0 z_{21}^{-1}(x_{i+1}-x_{i}) + c\eps^\zeta,\qquad\text{for $i$ odd,}\\
\int_{x_i}^{x_{i+1}} \bigl(\eps^4 u^2_{xx}+ \eps^{-2}W(u_x)+ \eps^{-2}u^2\bigr)\,\mathrm dx\leq B_0 z_{31}^{-1}(x_{i+1}-x_{i}) + c\eps^\zeta,\qquad\text{for $i$ even.}
\end{split}
\eeq
Furthermore, for every $\sigma\in(\eps^\frac{1}{\max\{3,q\}},\eta_0)$,
\beq
\label{da combino}
\begin{split}
\bigl| \mathscr{L}\bigl( (x_i,x_{i+1})\cap\{| u_x-z_2|\leq \sigma\} \bigr) - z_{21}^{-1}(x_{i+1}-x_i)\bigr| + \bigl| \mathscr{L}\bigl( (x_i,x_{i+1})\cap\{| u_x-z_3|\leq \sigma\} \bigr) \bigr|\leq c \eps^\zeta,
\\
\bigl| \mathscr{L}\bigl( (x_i,x_{i+1})\cap\{| u_x-z_2|\leq \sigma\} \bigr) \bigr| + \bigl| \mathscr{L}\bigl( (x_i,x_{i+1})\cap\{| u_x-z_3|\leq \sigma\} \bigr) - z_{31}^{-1}(x_{i+1}-x_i)\bigr|\leq c \eps^\zeta,
\end{split}
\eeq
respectively when $i$ is odd and $i$ is even.
}
\end{prpstn}
\begin{proof}
Here we generalise the approach devised in \cite{Muller1}. {For simplicity, we prove the statement assuming $n=3$ and $x_2=l_0$ for some $l_0\in(0,1).$ Let us also define $\lambda_2,\lambda_3$ as $\lambda_2:=z^{-1}_{21}l_0$ and $\lambda_{3}:=z_{31}^{-1}(1-l_0)$. We first construct the bit of $u$ with energy $A_0\lambda_{2}$ in $(0,l_0)$, and then use the same argument to construct on $(l_0,1)$ the bit of $u$ which has energy $B_0\lambda_3$.} 
\\

We start by splitting the interval $(0,l_0)$ into $N$ pieces of length $l_N:=\frac{z_{21}\lambda_{2}}{N}=\frac {l_0} N$. Let us also consider $\hat w (x)$, solution of 
\beq
\label{ODE}
\eps^3\hat w _x = \sqrt {W(\hat{w})},\qquad \hat w(0)=0.
\eeq
Standard ODE theory tells us that $\hat w$ exists, and that $\hat w$ is strictly increasing with $x$ when $\hat w(x)\in (z_1,z_2)$. We point out that, in case $q<2$, the solution might not be unique. In this case, when solutions encounter $z_1$ or $z_2$ we choose the one that stays bounded in $[z_1,z_2]$ and does not decrease/increase further. As $\hat w(x-\omega)$ still satisfies the equation in \eqref{ODE} for every $\omega\in\R$, we will choose $\omega=\omega^*$ so that 
\beq
\label{Fstar}
F(\omega^*):=\int_0^{l_N} \hat w(s-\omega^*)\,\mathrm d s = 0.
\eeq
Indeed, this is possible as $F$ is negative for $\omega\to\infty$, positive when $\omega\to-\infty$, continuous and decreasing. Now we define $w$ as
\[
w(x) = 
\begin{cases}
\hat w (x-\omega^*-x_i),\qquad &\text{if $x\in(x_i,x_{i+1})$ when $i$ even},\\
\hat w (x_{i+1}-\omega^*-x),\qquad &\text{if $x\in(x_i,x_{i+1})$ when $i$ odd},
\end{cases}
\]
where $x_i:=il_N$ for $i=0,\dots,N$. We are now ready to construct $u$ as
\beq
\label{uuu}
u(x):=\int_0^x w(s)\,\mathrm ds,
\eeq
and to notice that, by \eqref{Fstar}, $u(x_i)=0$ for each $i=0,\dots,N$. By \eqref{ODE} we have
$$
\int_{x_i}^{x_{i+1}} \bigl(\eps^4 u^2_{xx}+ \eps^{-2}W(u_x)\bigr)\,\mathrm dx = 2 \eps \int_{x_i}^{x_{i+1}} |W(u_x)|^{\frac12}|u_{xx}|\,\mathrm dx \leq 2 \eps\int^{z_2}_{z_1}|W(s)|^{\frac12}\mathrm{d}s =  \eps E_0.
$$
On the other hand, called $x^*_i$ the point in $(x_i,x_{i+1})$ such that $u_x(x^*_i)=0$, and assuming without loss of generality that $u_x>0$ in $(x_i,x_{i}^*)$ (the case $u_x<0$ is similar), we have
\beq
\label{energia u2 ref}
\int_{x_i}^{x_{i+1}} u^2\,\mathrm dx\leq \int_{x_i}^{x^*_i}\bigl(z_2 (x-x_i)\bigr)^2\mathrm dx+ \int_{x_i^*}^{x_{i+1}}\bigl(z_1 (x_{i+1}-x)\bigr)^2\mathrm dx = 3^{-1}\bigl(z_2^2\alpha^3 + z_1^2\gamma^3\bigr)l_N^3,
\eeq
where 
$$
\alpha:=\mathscr{L}\bigl((x_0,x_1)\cap \{w\geq 0 \}\bigr)l_N^{-1},\qquad \gamma:=\mathscr{L}\bigl((x_0,x_1)\cap \{w < 0 \}\bigr)l_N^{-1}.
$$
Therefore,
\beq
\label{stima sup 11}
\begin{split}
\int_{0}^{l_0} \bigl(\eps^4 u^2_{xx}+ \eps^{-2}W(u_x)+\eps^{-2}u^2\bigr)\,\mathrm dx\leq N\biggl(3^{-1}\eps^{-2}\bigl(z_2^2\alpha^3 + z_1^2\gamma^3\bigr)l_N^3+\eps E_0\biggr)\\
=  l_0 \biggl(3^{-1}\bigl(z_2^2\alpha^3 + z_1^2\gamma^3\bigr) d_\eps^2+ E_0 d_\eps^{-1}\biggr),
\end{split}
\eeq
where $d_\eps = \frac{l_N}{\eps}$. Now, chosen $\eta\in(0,\eta_0)$, with $\eta_0$ as in (H4), we notice that
\beq
\label{sup diff}
0 = \int_0^{l_N} \hat w(s-\omega^*)\,\mathrm ds\leq (z_2\alpha+ (z_1+\eta)\gamma)l_N + r, 
\eeq
with $r:=-(z_1+\eta)\mathscr{L}(\{s\colon \hat w(s-\omega^*)\in(z_1+\eta,0)\})$. But $r$ can be estimated as follows: we can rewrite \eqref{ODE} in terms of $\hat{v}:=\hat{w}-z_1$ as
$$
\eps^3\hat v_y = -\sqrt{W(\hat v + z_1)}, \qquad \hat v (0) = - z_1,
$$
where we also made the change of variable $y=-x$. Now, called $y^*$ the point in $\R_+$ where $\hat v(y^*)= \eta_0$, by (H5) we have
$$
\eps^3\hat v_y(y) \leq - \hat c,\qquad\text{for all $y\in(0,y^*],$}
$$
for some $\hat c>0$. After an integration in $y$ between $0$ and $y^*$, this leads to $y^*\leq c\eps^3.$ In the same way, when $\eta\leq\hat v<\eta_0$, (H5) implies
$$
\eps^3\hat v_y \leq - c_0|\hat v|^{\frac q2}\leq- c_0|\hat v|^{\frac {\max \{3, q\}} 2}.
$$
Let us now denote by $\tilde y$ the point in $\R^+$ such that $\hat v(\tilde y)=\eta$. An integration between $y^*$ and $\tilde y$ yields $\tilde y-y^*\leq c\eps^3\eta^{1-\frac {\max \{3, q\}} 2}.$ Thus, as $\mathscr{L}(\{s\colon \hat w(s-\omega^*)\in(z_1+\eta,0)\})= \tilde{y}$, we have obtained
\beq
\label{chenneso}
|r|\leq |z_1|\tilde y\leq c\eps^3\eta^{1-\frac {\max \{3, q\}} 2}.
\eeq
This together with \eqref{sup diff} thus imply
\beq
\label{stima eq 00}
\gamma \leq \frac{z_2}{|z_1|}\alpha +c\bar r ,
\eeq
where $\bar r := \eta + N\eps^3\eta^{1-\frac {\max \{3, q\}} 2}.$  
On the other hand, 
$$
0 = \int_0^{l} \hat w(s-\omega^*)\,\mathrm ds\geq ((z_2-\eta)\alpha+ z_1\gamma)l_N + r_1, 
$$
where now $r_1:=(z_2-\eta)\mathscr{L}(\{s\colon \hat w(s-\omega^*)\in(0,z_2-\eta)\})$. By arguing as to get \eqref{stima eq 00}, we have
$$
\frac{z_2}{|z_1|}\alpha\leq\gamma + c\bar r, 
$$
and, as $\alpha+\gamma = 1$, by \eqref{stima eq 00} we thus deduce $\bigl|\alpha - \frac{z_1}{z_1-z_2}\bigr| \leq  c\bar r.$ The fact that, by construction, $l_0\frac{z_1}{z_1-z_2} = \lambda_2$ also implies 
\beq
\label{equivalgono123}
\lambda_2 - c \bar r\leq \alpha l_0\leq\lambda_2 + c \bar r.
\eeq 
We now choose $N$ as the smallest even integer larger than $l_0(\eps d^*)^{-1}$, where $$d^*=(3E_0)^{\frac13}\Bigl(2\lambda_2^3 l_0^{-3} z_2^2\Bigl(1 - \frac{z_2}{z_1}\Bigr)\Bigr)^{-\frac13}.$$ In this way, $d_\eps\leq d^*$, $d_\eps^{-1}\leq c\eps + (d^*)^{-1}$ and $N\leq c\eps^{-1}$. Let us also choose $\eta = \eps^{\frac4{\max\{3,q\}}}$, so that $\bar{r}\leq c\eps^{\frac4{\max\{3,q\}}}$ and $\eps^{\frac4{\max\{3,q\}}}<1$ for each $\eps\leq\eps_0<1$. By exploiting \eqref{stima eq 00}--\eqref{equivalgono123} in \eqref{stima sup 11} we thus get
\beq
\label{res 1}
\begin{split}
\int_{0}^{l_0} \bigl(\eps^4 u_{xx}^2+ \eps^{-2}W(u_x)+\eps^{-2}u^2\bigr)\,\mathrm dx
&\leq l_0 \biggl(3^{-1}\alpha^3 z_2^2\bigl(1 - \frac{z_2}{z_1}\bigr) d_\eps^2+ E_0 d_\eps^{-1}\biggr)+ c\tilde{r}
\\
&\leq l_0 \biggl(3^{-1}\lambda_2^3l_0^{-3} z_2^2\bigl(1 - \frac{z_2}{z_1}\bigr) d_\eps^2+ E_0 d_\eps^{-1}\biggr)+ c\tilde{r}\\
&\leq l_0 \biggl(3^{-1}\lambda_2^3l_0^{-3} z_2^2\bigl(1 - \frac{z_2}{z_1}\bigr) (d^*)^2+ E_0 (d^*)^{-1}\biggr)+ c\tilde{r}
\\
&\leq \lambda_{2} A_0 + c\tilde{r}
.
\end{split}
\eeq
Here and below $\tilde{r}:=\eps^{\frac4{\max\{3,q\}}}+\eps$. We remark that $\alpha$ depends on $\eps$, but for every $\sigma\in(\eps^{\frac{1}{\max\{3,q\}}},\eta_0)$, we have that 
\beq
\label{eq 3star}
 \mathscr{L}\bigl( (0,l_0)\cap\{| u_x-z_2|\leq \sigma\} \bigr) = \alpha l_0 - R,  
 \eeq
where $R = N \mathscr{L}((x_0,x_1)\cap \{s\colon 0<w(s)<z_2-\sigma)\}$. By arguing as in the proof of \eqref{chenneso} with $\eta$ replaced by $\sigma$, we have that $|R|\leq c N\sigma^{1-\max\{3,q\}}\eps^3 \leq c \sigma^{-\max\{3,q\}}\eps^2$. Thus, since we assumed $\sigma\geq \eps^{\frac{1}{\max\{3,q\}}}$, we deduce $|R|\leq c\eps$. Therefore, by recalling \eqref{equivalgono123} with $\eta=\eps^{\frac{4}{\max\{3,q\}}}$, from \eqref{eq 3star} we finally obtain
\beq
\label{motescrivo} \bigl| \mathscr{L}\bigl( (0,l_0)\cap\{| u_x-z_2|\leq \sigma\} \bigr) - \lambda_{2}\bigr| \leq c \eps^{\frac{4}{\max\{3,q\}}}+c\eps.  
\eeq

\begin{figure}
\begin{subfigure}{.5\textwidth}
  \centering
  \includegraphics[width=.99\linewidth]{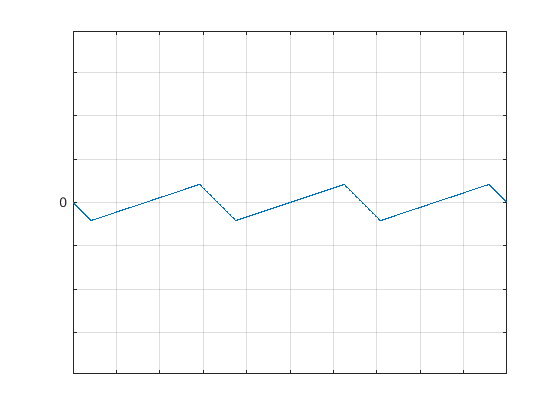}
  \caption{}
  \label{fig00001}
\end{subfigure}%
\begin{subfigure}{.5\textwidth}
  \centering
  \includegraphics[width=.99\linewidth]{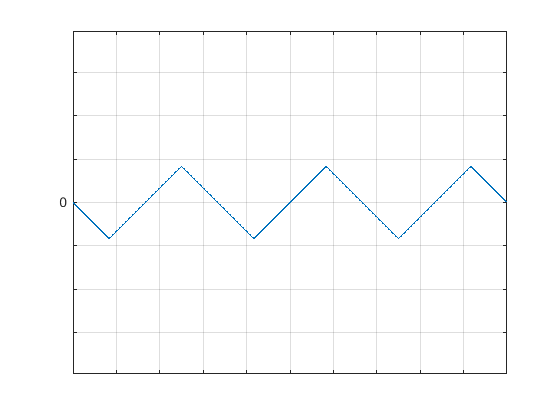}
  \caption{}
  \label{fig00002}
\end{subfigure}
\caption{Piecewise approximation of the constructed function: in Figure \ref{fig00001} we show the function constructed in $(0,l_0)$, whose gradient oscillates between $z_1$ and $z_2$. In Figure \ref{fig00002} we show the function constructed in $(l_0,1)$, whose gradient oscillates between $z_1$ and $z_3$.}
\end{figure}

Let us now focus on the interval $(l_0,1),$ where we want to construct the part of $u$ related to the $B_0-$term of the energy in \eqref{sup goal}. This part of the argument is very similar to the one above, but, as there might be no solution to \eqref{ODE} connecting $z_1$ to $z_3$, this time we need to construct an $u$ whose gradient is slightly more complicated. Below, we try to highlight the differences from the case above without incurring into many repetitions. Let us consider $\tilde{w}$ to be the solution to
\beq
\label{ODE2}
\eps^3\tilde w_x = \sqrt{W(\tilde w)},\qquad \tilde w(s_0+2\mu^{\theta+1})=z_2+\mu,
\eeq
where $s_0>0$ is such that $\hat w(s_0) = z_2-\mu$, $\hat{w}$ is as in \eqref{ODE}, and $\theta=\frac{3}{2}(\max\{q,3\}-2)$. Here and below $\mu = \eps^{\frac2{\max\{3,q\}-2}}$, so that $\mu^\theta=\eps^3$. We remark that an argument as the one to prove \eqref{chenneso} yields 
$$
s_0 \leq c\eps^3\mu^{1-\frac{\max\{3,q\}}{2}}\leq c\eps^2,
$$
so that $s_0$ does not explode but actually goes to zero faster than $\eps$. Again, if $q<2$ $\tilde{w}$ might not be unique, but we choose the one which stays bounded in $[z_2,z_3].$ Let us define $v$ as
\[
v(s)=
\begin{cases}
\hat w(s),\qquad&\text{if $s\leq s_0$},\\
\mu^{-\theta}(s - s_0)+(z_2-\mu),\qquad&\text{if $s_0 < s \leq s_0+2\mu^{\theta+1}$},\\
\tilde w(s),\qquad&\text{if $s_0 +2\mu^{\theta+1} < s$}.
\end{cases}
\]
Again, we divide $(l_0,1)$ into $M$ subintervals of equal length $l_M:=M^{-1}(1-l_0)$, and notice that, as $v$ is monotone, we can find $\omega_*$ such that 
$$
\int_0^{l_M} v(s-\omega_*)\,\mathrm{d}s = 0.
$$ 
As in the previous part of the proof, we construct
\[
w(x) = 
\begin{cases}
v (x-\omega_*-y_i),\qquad &\text{if $x\in(y_i,y_{i+1})$ when $i$ even and $i\neq 0$},\\
v (y_{i+1}-\omega_*-x),\qquad &\text{if $x\in(y_i,y_{i+1})$ when $i$ odd},
\end{cases}
\]
with $y_i:=l_0+ il_M$, for $i=0,\dots,M$, and $u$ as in \eqref{uuu}. We remark that, as in general 
$$
\lim_{s\uparrow l_0}w(s) = \hat w(-\omega^*)\neq\hat w(-\omega_*),
$$
$w$ needs to be defined differently in $(y_0,y_1)$ in order to be continuous and to have $u\in H^2(0,1).$ For this reason, we construct $w$ as follows in $(y_0,y_1)$:
\[
w(x) = 
\begin{cases}
\hat w (-\omega^*) + \frac{s-y_0}{\eps^3}(z_1-\hat w (-\omega^*)),\qquad &\text{if $x\in(y_0,y_{0}+\eps^3)$},\\
z_1,\qquad &\text{if $x\in(y_0+\eps^3, y_{0}+\eps^3 + a)$ },\\
z_1 + \frac{s - y_0-\eps^3-a}{\eps^3}(z_3-z_1),\qquad &\text{if $x\in(y_{0}+\eps^3 + a,y_{0}+2\eps^3 + a)$},\\
z_3,\qquad &\text{if $x\in(y_{0}+2\eps^3 + a, y_{1}-\eps^3)$ },\\
v (l_M-\omega_*) + \frac{y_1-s}{\eps^3}(z_3- v (l_M-\omega_*)),\qquad &\text{if $x\in(y_1-\eps^3, y_{1})$},\\
\end{cases}
\]
where $a$ is such that $\int_{y_0}^{y_1}w(s)\,\mathrm{d}s =0.$ We point out that such $a$ exists for each $\eps\leq\eps_0$, for some $\eps_0<1$ depending on $z_1$, $z_3$ and $l_0$ only. After defining $u$ in $(y_0,y_1)$ as in \eqref{uuu}, we have $u(y_0)=u(y_1)=0$ and
\begin{align*}
&\int_{y_0}^{y_1}u^2(s)\,\mathrm{d}s\leq \bigl(\max\{|z_1|,z_3\}\bigr)^2\int_{y_0}^{y_1}(s-y_0)^2\,\mathrm{d}s\leq c l_M^3,
\eps^{-3}.\\
\end{align*}
Thus,
$$
\int_{y_0}^{y_{1}} \bigl(\eps^4 u^2_{xx}+ \eps^{-2}W(u_x)+ \eps^{-2}u^2\bigr)\,\mathrm ds\leq c(\eps +\eps^{-2}M^{-3}).
$$
On the other hand, if $i>0$, by the definition of $E_0,E_1$ and by the way we constructed $u$ we have
\[
\begin{split}
\int_{y_i}^{y_{i+1}} \bigl(\eps^4 u^2_{xx}+ \eps^{-2}W(u_x)\bigr)\,\mathrm dx
\leq 2\eps\biggl( \int_{z_1}^{z_2}&|W(s)|^{\frac12}\,\mathrm d s+\int_{z_2}^{z_3}|W(s)|^{\frac12}\,\mathrm d s\biggr) + c\mu^{\theta+1}(\eps^4\mu^{-2\theta}+\eps^{-2})
\\
&\leq \eps( E_0+E_1)+ c\mu\eps.
\end{split}
\]
Furthermore, once defined
$$
\beta:=\mathscr{L}\bigl((y_1,y_2)\cap \{w\geq 0 \}\bigr)l_M^{-1},\qquad \gamma:=\mathscr{L}\bigl((y_1,y_2)\cap \{w < 0 \}\bigr)l_M^{-1},
$$
by arguing as in the proof of \eqref{energia u2 ref} we deduce
$$
\int_{y_i}^{y_{i+1}}u^2\leq 3^{-1} \bigl(z_3^2\beta^3+z_1^2\gamma^3 \bigr)l_M^3.
$$
Therefore, collecting the inequalities above
\beq
\label{last min}
\begin{split}
\int_{l_0}^{1} \bigl(\eps^4 u^2_{xx}&+ \eps^{-2}W(u_x)+\eps^{-2}u^2\bigr)\,\mathrm dx \\
&\leq c(\mu\eps M+\eps +\eps^{-2}M^{-3})
+(1-l_0)\Bigl((E_0+E_1)h_\eps^{-1} + 3^{-1}(z_3^2\beta^3+z_1^2\gamma^3)h_\eps^2\Bigr), 
\end{split}
\eeq
where now $h_\eps=\frac{l_M}{\eps}$. As in \eqref{sup diff} we have
\beq
\label{torefertoN}
0 = \int_0^{l_M}v(s-\omega^*)\,\mathrm{d}s\geq (z_1\gamma+(z_3-\mu)\beta)l_M-r_2,
\eeq
with $r_2 : = (z_3-\mu) \mathscr{L}\bigl(\{s:v(s-\omega^*)\in(0,z_3-\mu)\} \bigr)$.
We first notice that 
$$
r_2 = (z_3-\mu)\Bigl(\mu^{\theta+1} + s_0 +  \mathscr{L}\bigl(\{s:\tilde w(s)\in(z_2+\mu,z_3-\mu)\} \bigr)\Bigr).
$$
Thus, by arguing as in the proof of \eqref{chenneso} we first deduce
$$
\mathscr{L}\bigl(\{s:\tilde w(s)\in(z_2+\mu,z_3-\mu)\} \bigr)\leq c\eps^3\mu^{1-\frac{\max\{3,q\}}{2}},
$$
and therefore
\beq
\label{bene 000}
|r_2|\leq c(\eps^3\mu^{1-\frac{\max\{3,q\}}{2}}+s_0+\mu^{\theta+1})\leq c\eps^2.
\eeq
Define $\bar r_M:=M \eps^2 +\mu$, then \eqref{torefertoN}--\eqref{bene 000} imply
\beq
\label{towrite2}
\frac{z_3}{|z_1|}\beta\leq\gamma+c\bar{r}_M.
\eeq
In the same way, we can prove that
\beq
\label{towrite}\gamma\leq \frac{z_3}{|z_1|}\beta +c\bar{r}_M,
\eeq
and, recalling that $\beta+\gamma=1$, $(1-l_0)\frac{z_1}{z_1-z_3} = \lambda_3$, by \eqref{towrite2} we obtain
\beq
\label{hurahota}
\lambda_3-c\bar{r}_M\leq \beta(1-l_0)\leq \lambda_3+c\bar{r}_M.
\eeq
Then, after choosing $M$ to be the smallest integer larger than $(1-l_0)(\eps h^*)^{-1}$, with $$h^*:= \bigl(3(E_0+E_1)\bigr)^{\frac13}\Bigl(2\lambda_3^3(1-l_0)^{-3}z_3^2\Bigl(1-\frac{z_3}{z_1}\Bigr)\Bigr)^{-\frac13}, $$
and exploiting \eqref{towrite}--\eqref{hurahota}, \eqref{last min} becomes
\[
\begin{split}
\int_{l_0}^1\bigl(\eps^4 u^2_{xx}+ \eps^{-2}W(u_x)+\eps^{-2}u^2\bigr)\,\mathrm dx
&\leq (1-l_0)\Bigl((E_0+E_1)h_\eps^{-1} + 3^{-1}z_3^2\lambda_3^3(1-l_0)^{-3}\bigr(1-\frac{z_3}{z_1}\Bigr)h_\eps^2\Bigr)+c\hat r\\
&\leq \beta B_0 (1-l_0) + c\hat r\leq B_0\lambda_3 + c\hat r,
\end{split}
\]
where $\hat r = \eps + \mu$. Here, we repeatedly used the fact that $M \leq c\eps^{-1},M^{-1} \leq c\eps$ and that $\mu,\eps<1$ in order to estimate the above error. This together with \eqref{res 1} proves \eqref{sup goal}.\\

\noindent 
Now, since $\sigma>\eps^{\frac{1}{\max\{3,q\}}}>\mu$,
\[
\begin{split}
 \mathscr{L}\bigl( (l_0,1)\cap\{| u_x-z_2|\leq \sigma\} \bigr) \leq  
\mathscr{L}\bigl( (l_0,1)\cap\{| u_x-z_2|\leq \mu\} \bigr)+\mathscr{L}\bigl( (l_0,1)\cap\{\mu\leq | u_x-z_2|\leq \sigma\} \bigr)\\
\leq cM(\mu^{\theta+1}+\eps^3\sigma^{1-\frac{\max\{3,q\}}{2}})+c\eps^3\leq
 c\eps,  
\end{split}
\]
where we argued as to get \eqref{chenneso} in order to bound $\mathscr{L}\bigl( (l_0,1)\cap\{\mu\leq | u_x-z_2|\leq \sigma\} \bigr).$
Furthermore,
$$
\mathscr{L}\bigl((l_0,1)\cap\{|u_x-z_3|\leq \sigma\}\bigr) = \beta(1-l_0)+R_2 -\beta l_M+\mathscr{L}\bigl((y_0,y_1)\cap\{|u_x-z_3|\leq \sigma\}\bigr),
$$
where $R_2:=(M-1)\mathscr{L}\bigl((y_1,y_2)\cap\{v(s-\omega^*)\in(0,z_3- \sigma)\}\bigr)$. As $R_2\leq c M |r_2|\leq c\eps$ and $l_M\leq c\eps$ we have 
\beq
\label{forselast}
\bigl|\mathscr{L}\bigl((l_0,1)\cap\{|u_x-z_3|\leq \sigma\}\bigr) - \beta(1-l_0)\bigr|\leq c\eps.
\eeq
{Now, recalling that \eqref{hurahota} implies $|\lambda_3-\beta(1-l_0)|\leq  c\hat r$, \eqref{forselast} and the triangular inequality imply
$$ \bigl| \mathscr{L}\bigl( (l_0,1)\cap\{| u_x-z_3|\leq \sigma\} \bigr) - \lambda_3\bigr| \leq c\hat r.  $$
This together with \eqref{motescrivo} lead to the second statement of the result.}
\end{proof}

\section{Construction of a lower bound}
\label{section 4}
{This section is the core of this paper, and is where we prove a lower bound for the energy depending on the global volume fractions $\lambda_k^{\eta}(v)$ defined in \eqref{volfrac} below, and representing the one-dimensional Lebesgue measure of the set where $v_x$ is in an $\eta-$neighbourhood of $z_k$. Here, $\eta\in (\eps^{\frac{1}{q+1}},\eta_0)$, $\eta_0$ is as in (H4) and $v\in V$. We point out that the presence of a third well gives the possibility of many different microstructures (see e.g., Figure \ref{fig00001}, Figure \ref{fig00002} and Figure \ref{fig:fig}), and makes the estimates below long and technical. 

The strategy to prove our lower bounds is the following: for every $v\in V$ of finite energy we identify $L-$intervals (see Definition \ref{defDC}), sets in which $v_x>z_1+\eta$ and containing a subset of positive measure where $v_x>z_2+\eta$. By Lemma \ref{numero} below, the number $N_v$ of $L-$intervals is finite, and can be bounded by a constant times $\frac1\eps.$ 
In Proposition \ref{sotto} we estimate from below the $L^2-$norm of $v$ in the $L-$intervals $L_i\subset (0,1)$, with $i=1,\dots,N_v$. We highlight that the sharp estimates are different for different types of microstructures (see Definition \ref{GliDIntervals} and Figure \ref{ClassPict}). 
We then identify (possibly empty) regions $\Sigma_i\subset(0,1)$ in the set where $v_x$ is in an $\eta-$neighbourhood of $z_1$, and in these sets we estimate the $L^2-$norm of $v$. The lower bounds for the $L^2-$norm in the sets $\Sigma_i$ are combined with the $L^2-$estimates in the $L_i$'s to obtain good lower bounds for the $L^2-$norm of $v$ on every disjoint set $F_i:=L_i\cup\Sigma_i$.
The interface energy, that is, the energy necessary for the transition of $v_x$ from one well of $W$ to another, can be bounded via the Modica-Mortola estimate
$$
\int_{a}^{b} \bigl( \eps^4 v_{xx}^2+\eps^{-2}W(v_x)\bigr)\,\mathrm dx\geq 2 \eps\int_{a}^{b} |\sqrt{W(v_x)}v_{xx}|\,\mathrm dx \geq  2 \eps\Bigl| \int_{v_x(a)}^{v_x(b)} \sqrt{W(s)}\,\mathrm ds\Bigr|,
$$
valid for every $0\leq a\leq b\leq 1.$ We show that for each $i=1,\dots,N_v$
\beq
\label{minimideps}
\begin{split}
\text{energy of v in $F_i$} + \text{small error} \geq \eps\text{(interface energy in $F_i$)} + \frac{1}{\eps^2}\text{($L^2-$ norm of $v$ in $F_i$)} \\
\geq
\min_{d>0}\Bigl\{d\text{(interface energy in $F_i$)} + \frac{1}{d^2}\text{($L^2-$ norm of $v$ in $F_i$)}\Bigr \}.
\end{split}
\eeq
{In the two-well case (see \cite{Muller1}), it is possible to sum the resulting lower bounds over $i=1,\dots,N_v$, and to obtain a lower bound depending on global quantities only.} In our case, however, the lower bounds deduced via \eqref{minimideps} are nonlinear in the volume fractions $\ai,\bi$ (see \eqref{defaibi} below), defined respectively as the Lebesgue measures of the regions of $F_i$ where $v_x$ is close to $z_2,z_3$. Furthermore, we get lower bounds which are different depending on the different microstructures in the interval (see e.g., Figure \ref{fig00001}, Figure \ref{fig00002} and Figure \ref{fig:fig}). This means that different microstructures give a different dependence of the lower bound on the volume fractions $\ai,\bi$. These facts increase the complexity of the problem, as they do not allow one, in general, to collect the estimates for the different $F_i$ and to obtain a lower bound depending only on the global volume fractions $\lambda_k^\eta(v)$, $k=1,2,3$. 

Finally, in Theorem \ref{bound da sotto thm} we use assumptions (H6)--(H8) to bound from below the estimates obtained in Proposition \ref{sotto} with the linear function $A_0\ai+B_0 \bi$. 
We can hence sum the contribution of every disjoint set $F_i$ and obtain the final lower bound $A_0\lambda_2^{\eta}(v)+B_0\lambda_3^{\eta}(v)$. The final estimate looks independent of $\lambda_1^{\eta}(v)$, but this is because we implicitly make use of
$$
\sum_{k=1}^3\lambda_k^{\eta}(v) = 1 + \text{small error}.
$$
}

Let $\eta_0>0$ be as in (H4). Given a generic $v\in H^2(0,1)$, $\eta\in(0,\eta_0)$ let us define the $k$-th global volume fraction for $v$ as
\beq
\label{volfrac}
\lambda_k^{\eta}(v):=  \mathscr{L}\bigl(\bigl\{x\in (0,1)\colon |v_{x}(x)-z_k|\leq\eta\bigr\}\bigr),\qquad k=1,2,3,
\eeq
and let us also generalize the definition of transition layers given in \cite{Muller1} (cf. also Figure \ref{AB tl})
\begin{dfntn}
Let $v\in H^2(a,b)$ and $\eta\in(0,\eta_0)$. An interval $(x^-,x^+)$ is called an $A_+^{\eta}-$transition (resp. an $A_-^{\eta}-$transition) layer for $v$ if
\begin{align*}
&v_x(x)  \in (z_1+\eta, z_2-\eta), \qquad\forall x\in(x^-,x^+),\\
&v_x(x^-) = z_1+\eta, \qquad(\text{resp. $v_x(x^+) = z_1+\eta$}),\\
&v_x(x^+) = z_2-\eta, \qquad(\text{resp. $v_x(x^-) = z_2-\eta$}).
\end{align*}
An interval $(x^-,x^+)$ is called a $B_+^{\eta}-$transition (resp. a $B_-^{\eta}-$transition) layer for $v$ if
\begin{align*}
&v_x(x)  \in (z_2+\eta, z_3-\eta), \qquad\forall x\in(x^-,x^+),\\	
&v_x(x^-) = z_2+\eta, \qquad(\text{resp. $v_x(x^+) = z_2+\eta$}),\\
&v_x(x^+) = z_3-\eta, \qquad(\text{resp. $v_x(x^-) = z_3-\eta$}).
\end{align*}
\end{dfntn}
Given a function $v\in H^2(0,1)$ and $\eta\in(0,\eta_0)$ we denote by $\# A^\eta_+$ (or by $\# A^\eta_-,\# B^\eta_+,\# B^\eta_-$) the number of $A_+^{\eta}-$transition layers for $v$ (resp. $\# A^\eta_-,\# B^\eta_+,\# B^\eta_--$transition layers for $v$) in the interval $(0,1)$. {The number of transition layers of a function $v$ with bounded energy can be controlled by a constant times $\eps^{-1}$, as stated in the following lemma:}
\begin{lmm}
\label{numero}
Assume {\rm(H1)--(H5)}, $\eta < \eta_0$, and let $\eps>0$ and $v\in H^2(0,1)$ be such that $I^\eps(v)\leq C$. Then, there exists $c=c(C)>0$ such that
$$
\max \bigl\{ \# A_+^{\eta},\# A_-^{\eta},\# B_+^{\eta},\# B_-^{\eta}\bigr\}\leq c\eps^{-1}.
$$
\end{lmm}
\begin{proof}
Let us first recall that, given $0\leq a\leq b\leq 1$ we have
\beq
\label{MMest}
\int_{a}^{b} \bigl( \eps^4 v_{xx}^2+\eps^{-2}W(v_x)\bigr)\,\mathrm dx\geq 2 \eps\int_{a}^{b} |\sqrt{W(v_x)}v_{xx}|\,\mathrm dx\geq  2\eps \big|H(v_x(b))-H(v_x(a))\big| ,
\eeq
where $H(s) = \int_0^s\sqrt{W(r)}\,\mathrm{d}r$. Now, let us restrict ourselves to the case of the $A_\pm^{\eta}-$transition layers, as the proof for the $B_\pm^{\eta}-$transition layers follows the same strategy. Let $(x^-,x^+)$ be an $A_\pm^{\eta}-$transition layer, then by \eqref{MMest} and the fact that $\eta<\eta_0$ we have 
\beq
\label{boundABt}
\int_{x^-}^{x^+} \bigl( \eps^4 v_{xx}^2+\eps^{-2}W(v_x)\bigr)\,\mathrm dx\geq   2\eps (H(z_2-\eta_0)-H(z_1+\eta_0)) > \eps \tilde c,
\eeq
for some positive constant $\tilde c$. 
Summing all the $A_\pm^{\eta}-$transition layers we thus get
$$
C\geq I^\eps(v)\geq \eps\tilde{c}(\# A_+^{\eta}+\# A_-^{\eta}),
$$
which concludes the proof.
\end{proof}

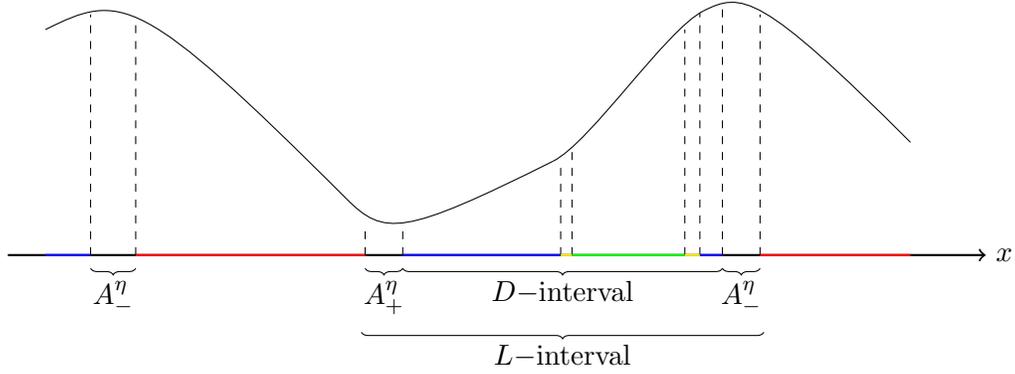
\begin{figure}
\centering
\begin{tikzpicture}[scale=0.50]
\draw (7,4) .. controls (9,5) and (9.5,5) .. (15,-0.5);\draw (15,-0.5) .. controls (16,-1.5) and (16.5,-1.5) .. (20.5,0.5);
\draw (20.5,0.5) .. controls (21.5,1) and (23.5,4) .. (24.5,4.5);
\draw (24.5,4.5) .. controls (25.5,5) and (26,5) .. (30,1);
\draw [dashed,thin] (8.2,-2) -- (8.2,4.4);
\draw [dashed,thin] (9.4,-2) -- (9.4,4.2);
\draw [dashed,thin] (15.5,-2) -- (15.5,-1.25);
\draw [dashed,thin] (16.5,-2) -- (16.5,-1.35);
\draw [dashed,thin] (20.7,-2) -- (20.7,0.5);
\draw [dashed,thin] (21,-2) -- (21,0.8);
\draw [dashed,thin] (24.0,-2) -- (24.0,4);
\draw [dashed,thin] (24.4,-2) -- (24.4,4.45);
\draw [dashed,thin] (25,-2) -- (25,4.73);
\draw [dashed,thin] (26,-2) -- (26,4.4);
\draw [->,thick] (6,-2) -- (32,-2);
\filldraw [red] (32,-2) circle (0pt) 
node[anchor=west,black] {$x$};

\draw[decoration={brace,mirror,raise=5pt},decorate]
  (8.2,-2) -- node[below=6pt] {$A_-^\eta$} (9.4,-2);
\draw[decoration={brace,mirror,raise=5pt},decorate]
  (15.5,-2) -- node[below=6pt] {$A_+^\eta$} (16.5,-2);
\draw[decoration={brace,mirror,raise=5pt},decorate]
  (25,-2) -- node[below=6pt] {$A_-^\eta$} (26,-2); 

\draw[decoration={brace,mirror,raise=5pt},decorate]
  (16.5,-2) -- node[below=6pt] {$D-$interval} (25,-2);
\draw[decoration={brace,mirror,raise=5pt},decorate]
  (15.4,-3.7) -- node[below=6pt] {$L-$interval} (26.1,-3.7);
\draw [red,thick] (9.4,-2) -- (15.5,-2);
\draw [red,thick] (26,-2) -- (30,-2);
\draw [blue,thick] (16.5,-2) -- (20.7,-2);
\draw [blue,thick] (24.4,-2) -- (25,-2);
\draw [blue,thick] (7,-2) -- (8.2,-2);
\draw [green,thick] (21,-2) -- (24,-2);
\draw [yellow,thick] (20.7,-2) -- (21,-2);
\draw [yellow,thick] (24,-2) -- (24.4,-2);
 \end{tikzpicture}
 \caption{\label{AB tl}Example of $L-$interval and $D-$interval defined in Definition \ref{defDC}. In this picture, the red, the blue and the green intervals are respectively the sets of points where $|v_x-z_1|\leq \eta,$ $|v_x-z_2|\leq \eta,$ and $|v_x-z_3|\leq \eta$. The $B^\eta_\pm-$transition layers are coloured in yellow.}
\end{figure}

We can now introduce also the $D-$intervals, which are the intervals between an $A_+^\eta-$transition layer $(y^-,x^-)$ and the first $A_-^\eta-$transition layer $(y^+,x^+)$ in order of appearance in $(0,1)$ after $(y^-,x^-)$ (see Figure \ref{AB tl}):
\begin{dfntn}
\label{defDC}
{
Let $v\in H^2(a,b)$ and $\eta\in(0,\eta_0)$.
Let $(y^-,x^-)$ be an $A_+^\eta-$transition layer for $v$ and $(y^+,x^+)$ be an $A_-^\eta-$transition layer for $v$, with $x^-\leq x^+.$ We say that $(y^-,y^+)$ is a $L-$interval for $v$, if 
\beq
\label{ladefdiD}
v_x(x)>z_1+\eta,\quad \text{ for each $x\in(y^-,y^+)$,}\qquad \text{ and }\qquad v_x(y^+)=v_x(y^-)=z_1+\eta.
\eeq 
If \eqref{ladefdiD} holds, the interval $(x^-,x^+)$ is called a $D-$interval for $v$.
 } 
\end{dfntn}
{It is important to notice that $v_x$ might take negative values in a $D-$interval. For every $v\in V$, the number 
$$N_v = \text{number of $D-$intervals for $v$ in $(0,1)$},$$
is finite. Indeed, for every $v\in V$, $v_x$ is continuous and $N_v$ is equal to the number of $A_+^{\eta}-$transition layers, which is finite by Lemma \ref{numero}. We denote by $D_i$ the $i$-th $D-$interval in order of appearance in the interval $(0,1)$, where $i$ goes from $1$ to $N_v$. This means that, given two $D-$intervals $D_i=(x_i^-,x_i^+)$ and $D_j=(x_j^-,x_j^+)$, we have $x_i^+<x_j^-$ if and only if $i<j$.} The same can be done for $L-$intervals. Given $v\in H^2(0,1)$ we define also the following quantities
\beq
\label{defaibi}
\alpha_i^{\eta}(v) :=  \mathscr{L}\bigl( \bigl\{x\in D_i\colon |v_x(x)-z_2|\leq\eta \bigr\}\bigr),\qquad
\beta_i^{\eta}(v) := \mathscr{L}\bigl( \bigl\{x\in D_i\colon |v_x(x)-z_3|\leq\eta \bigr\}\bigr).
\eeq
measuring the subset of $D_i$ where $v_x$ is respectively close to $z_2$ and $z_3$. For ease of notation, we omit the dependence on $v$ of $\ai^\eta,\bi^\eta$ and $\lambda_k^\eta$. In what follows we will also drop the $\eta$ from $\alpha_i^{\eta},\beta_i^{\eta}$, keeping their dependence from this variable implicit. We remark that, denoting $D_1 = (x_1^-,x_1^+)$, $D_{N_v} = (x_{N_v}^-,x_{N_v}^+)$, we have
\beq
\label{sommediaibi}
\begin{split}
\sum_{i=1}^{N_v} \alpha_i + \mathscr{L}\bigl( \bigl\{x\in (0,x_1^-)\cup(x_{N_v}^+,1)\colon |v_x(x)-z_2|\leq\eta \bigr\}\bigr)= \lambda^\eta_2,
\\ \sum_{i=1}^{N_v} \beta_i +\mathscr{L}\bigl( \bigl\{x\in (0,x_1^-)\cup(x_{N_v}^+,1)\colon |v_x(x)-z_3|\leq\eta \bigr\}\bigr) = \lambda^\eta_3,
\end{split}
\eeq
and that, in general, $\ai+\bi <\mathscr{L}(D_i).$ Below, we estimate the energy of a generic $v\in V$ on every $D-$interval in terms of the quantities $\ai,\bi$. In order to do that, we first need to prove the following lemma, which is graphically explained in Figure \ref{StimaBasso}
\begin{figure}
\centering
\begin{tikzpicture}[scale=0.7]
\draw [thin,blue] (0,0) -- (3,0);
\draw [thin,blue] (3,0) -- (7,2);
\draw [thin,blue] (7,2) -- (10,5);

\draw [thin,red] (0,0) -- (1,0.5);
\draw [thin,red] (1,0.5) -- (2,1.5);
\draw [thin,red] (2,1.5) -- (4,1.5);
\draw [thin,red] (4,1.5) -- (6,2.5);
\draw [thin,red] (6,2.5) -- (7,3.5);
\draw [thin,red] (7,3.5) -- (8,4);
\draw [thin,red] (8,4) -- (9,4);
\draw [thin,red] (9,4) -- (10,5);
\draw [->,thick] (-1,0) -- (11,0);
\filldraw [red] (11,0) circle (0pt) 
node[anchor=west,black] {$x$};
\filldraw [red] (10,0) circle (0pt) 
node[anchor=north,black] {$b$};
\filldraw [red] (0,0) circle (0pt) 
node[anchor=north,black] {$a$};
\draw [thin,dashed] (10,0) -- (10,5);
\draw [thin,dashed] (0,0) -- (0,0.1);
\draw [thin,dashed] (7,0) -- (7,2);
 \end{tikzpicture}
\caption{\label{StimaBasso} Representation of Lemma \ref{dasottointegrale}. In blue the optimal-function $u\in W^{1,1}(a,b)$, which is convex, in red a generic function $v\in W^{1,1}(a,b)$ such that $u(a)=v(a)$ and \eqref{speclevset} holds.}
\end{figure}
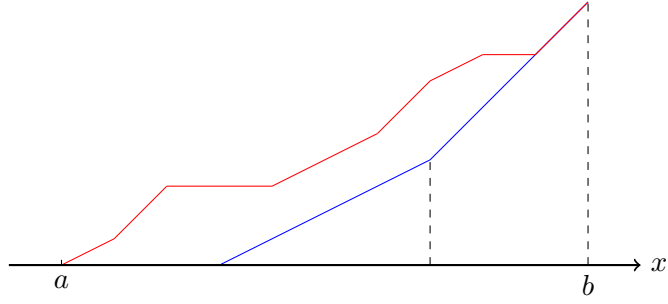
\begin{lmm}
\label{dasottointegrale}
Let $0\leq a\leq b \leq 1$ and let $u,v\in W^{1,1}(a,b)$ be two non-decreasing functions such that $u(a)=v(a)$ and
\beq
\label{speclevset}
\mathscr{L}\bigl(\{x\in(a,b)\colon u_x\geq \rho \}\bigr) = \mathscr{L}\bigl(\{x\in(a,b)\colon v_x\geq \rho \}\bigr),
\eeq
for every $\rho\geq 0$. If $u$ is the optimal function, that is if $u_x$ is non-decreasing in $(a,b)$, then $u(x)\leq v(x)$ for every $x\in[a,b]$.
\end{lmm}
\begin{proof}
We first notice that, as $u_x$ is non decreasing, $\{u_x\geq\rho\}$ is either empty, or an interval containing $b$. Thus, for every $ x\in(a,b)$, we have
\[
\mathscr{L}\bigl(\{u_x\geq \rho \}\cap(a,x)\bigr) = \bigl(\mathscr{L}\bigl(\{u_x\geq \rho \}\bigr)+x-b	\bigr)_+ = \bigl(\mathscr{L}\bigl(\{v_x\geq \rho \}\bigr)+x-b	\bigr)_+
\leq \mathscr{L}\bigl(\{v_x\geq \rho \}\cap(a,x)\bigr),
\]
where we denoted by $(\cdot)_+ := \max\{0,\cdot\}$, and where we used $\mathscr{L}(A\cap B) = \mathscr{L}(A)+\mathscr{L}(B)-\mathscr{L}(A\cup B)$ in the first and last passage. Therefore,
$$
u(x)- u(a) =  \int_a^x u_x(s)\,\mathrm{d}s = \int_0^\infty \mathscr{L}\bigl(\{u_x\geq \rho \}\cap(a,x)\bigr)\,\mathrm{d}\rho \leq \int_0^\infty \mathscr{L}\bigl(\{v_x\geq \rho \}\cap(a,x)\bigr)\,\mathrm{d}\rho
= v(x) -v(a) ,
$$
for every $x\in[a,b]$. As $u(a)=v(a),$ the claimed is proved.
\end{proof}
\begin{rmrk}
\label{resdasottoint rem}
It follows from Lemma \ref{dasottointegrale} that, given $0\leq a<b\leq 1$, and two Borel sets $\mathcal{C}_1,\mathcal{C}_2\subset (a,b)$ such that $\mathcal{C}_1\cap \mathcal{C}_2=\varnothing,$ then
\[
\int_a^b\Bigl(\tau_0+\sum_{i=1,2}\tau_i\mathscr{L}\bigl(	(a,x)\cap \mathcal{C}_i	\bigr)	\Bigr)^2\,\mathrm{d}x\geq  \int_0^{\mathscr{L}(\mathcal{C}_1)}(\tau_0+\tau_1x)^2\,\mathrm{d}x+\int_0^{\mathscr{L}(\mathcal{C}_2)}(\tau_0+\tau_1\mathscr{L}(\mathcal{C}_1)+\tau_2x)^2\,\mathrm{d}x
\]
for any $\tau_0\geq 0,$ $\tau_2>\tau_1\geq 0$.
\end{rmrk}
{We can now start to estimate the energy in the $L-$intervals. We start by obtaining the desired lower bound for all the $L-$intervals in which $\ai$ and $\bi$ are either too small or too large.}
\begin{lmm}
\label{corto}
Assume {\rm(H1)--(H5)}, and let 
$v\in H^2(0,1),$ $\eta\in(0,\eta_0)$ and $\eps\leq\eta^q$ be such that $I^{\eps}(v)\leq C$, with $C>0$. Then, there exist $R_*,R^*(C)>0$ such that for any $L_i=(y_i^-,y^+_i)$, with $\max\{\ai,\bi\}\geq R^*\eps$ or $\max\{\ai,\bi\}\leq R_*\eps$, 
\beq
\label{maxcontrol}
\int_{y_i^-}^{y^+_i}\bigl(\eps^4v_{xx}^2+\eps^{-2}W(v_x)+\eps^{-2}v^2\bigr)\,\mathrm dx\geq A_0\ai+B_0\bi.
\eeq
\end{lmm}
\begin{proof}
{First we want to prove that if either $\ai$ or $\bi$ is too large, then $v$ also becomes large, and hence its $L^2-$norm on $D_i$ is bigger than $A_0\ai+B_0\bi$}. In order to do this, let $\max\{\ai,\bi\}=R\eps$ for some $R>0$. Let $D_i=(x^-_i,x^+_i)$. We assume the existence of $x_i^*\in(x^-_i,x^+_i)$ such that $v(x_i^*)=0$, but the following estimates hold in the case $v>0$ (or $v<0$) in $(x^-_i,x^+_i)$ by taking $x_i^*=x^-_i$ (resp. $x_i^*=x^+_i$). Assume also without loss of generality that
\beq
\label{circle}
\max\bigl\{\mathscr{L}\bigl((x_i^*,x_i^+)\cap\{ |v_x-z_2|\leq\eta\}\bigr),\mathscr{L}\bigl((x_i^*,x_i^+)\cap\{ |v_x-z_3|\leq \eta\}\bigr)\bigr\}\geq \eps\frac R2,
\eeq
the alternative case can be proved similarly by replacing below $(x_i^*,x_i^+)$ with $(x_i^-,x_i^*)$. {We now approximate from below $v$ in $(x_i^*,x_i^+)$ with a piecewise linear function minus a small error proportional to $\eps$. Later, we use Lemma \ref{dasottointegrale} to estimate the $L^2-$norm of $v$ from below with the $L^2-$norm of its piecewise linear lower bound.} We remark that, as $(y_i^-,y_i^+)$ is a $L-$interval for $v$, $v_x(x)> z_1+\eta$ for each $x\in (y_i^-,y_i^+)$. Therefore, 
\beq
\label{aaa000}
v(x) \geq \int_{x_i^*}^x v_x\,\mathrm dx
\geq (z_2-\eta)\mathscr{L}\bigl((x_i^*,x)\cap\bigl\{v_x-z_2\geq -\eta\bigr\}\bigr)+
(z_1+\eta)\mathscr{L}\bigl( (x_i^*,x)\cap \{z_1+\eta<v_x\leq 0\}\bigr).
\eeq
{The last term in \eqref{aaa000} can be controlled from below by $(z_1+\eta)\mathscr{L}(\Sigma^\eta)$, }
where 
\beq\label{sigmadelta}\Sigma^\eta:= \{x\in (0,1)\colon|v_x-z_k|>\eta,\,\forall k=1,2,3\}.\eeq 
Thus, by the boundedness of $I^\eps(v)$ and (H5), we can write
$$
c_0\mathscr{L}\bigl(\Sigma^\eta\bigr)\eta^q\leq \int_ {\Sigma^\eta}W(v_x)\,\mathrm dx\leq \int_ 0^1W(v_x)\,\mathrm dx\leq C \eps^2,
$$
which implies
\beq
\label{dacitarepoi 000}
\mathscr{L}\bigl(\Sigma^\eta\bigr)\leq c \eps^2 \eta^{-q}.
\eeq
It follows then from \eqref{aaa000}, \eqref{dacitarepoi 000} and $\eps\leq \eta^q$ that
\beq
v(x) \geq \hat c\mathscr{L}\bigl((x_i^*,x)\cap\bigl\{v_x-z_2\geq -\eta\bigr\}\bigr)-c\eps,
\eeq
for every $x\in (x_i^*,x_i^+)$ and some positive constant $\hat c$. Therefore, thanks to \eqref{circle} and Lemma \ref{dasottointegrale},
{
\[
\begin{split}
2\int_{x_i^*}^{x_i^+} v^2\,\mathrm dx + c \eps^2(x_i^+-x_i^*)&\geq
\hat c^2\int_{x_i^*}^{x_i^+} \Bigl( \mathscr{L}\bigl((x_i^*,x)\cap\bigl\{v_x-z_2\geq -\eta\bigr\}\bigr)\Bigr)^2\mathrm{d}x
\geq \hat c^2\int_{0}^{\eps\frac R2} x^2 \,\mathrm dx
\geq \frac{\hat c^2}8\eps^3 R^3,  \end{split}
\]
}
which, by using the fact that
\beq
\label{limito intervallo}
(x_i^+-x_i^*)\leq 2R\eps+\mathscr{L}(\Sigma^\eta)\leq c\eps(1+R),
\eeq
yields
\beq
\label{dasapere}
\int_{y_i^-}^{y^+_i}\bigl(\eps^4v_{xx}^2+\eps^{-2}W(v_x)+\eps^{-2}v^2\bigr)\,\mathrm dx\geq \eps^{-2}\int_{x_i^-}^{x^+_i}v^2\,\mathrm dx\geq\bar c\eps R^3-c\eps(1+R),
\eeq
for some $\bar c>0.$ On the other hand, there exists $c^*>0$ such that
\beq
\label{chestima}
A_0\ai+B_0\bi\leq c^* R\eps.
\eeq
Therefore, setting $R^*$ as the biggest root of $\bar cR^3 = (c+c^*)(R+1)$, we deduce that, if $R\geq R^*$, \eqref{dasapere}--\eqref{chestima} imply \eqref{maxcontrol}. \\

In order to show that $R$ cannot be too small, we recall that in every $L-$interval there are exactly two $A^\eta_\pm-$ transition layers. By \eqref{MMest}--\eqref{boundABt} we thus have
\[
\begin{split}
\int_{y_i^-}^{y^+_i}\bigl(\eps^4v_{xx}^2+\eps^{-2}W(v_x)+\eps^{-2}v^2\bigr)\,\mathrm dx&\geq \int_{y_i^-}^{y_i^+}
\geq \eps\tilde{c},
\end{split}
\]
for some $\tilde{c}>0$. Hence, if we set $R_*=\frac{\tilde c} {c^*}$, by \eqref{chestima} we deduce that \eqref{maxcontrol} holds for every $R<R_*$. 
We remark that $R^*,R_*$ do not depend on $i,v,\eps,\eta$ in any way. $R_*$ does not depend on $C$ either.
\end{proof}
Let $n_i$ be the even number of $B_\pm^{\eta}-$transition layers for $v$ in $D_i$. If $n_i=2$, we denote the $B_+^{\eta}$ and the $B_-^{\eta}-$transition layers for $v$ in $D_i$ respectively by $(z_{i,1}^-,z_{i,1}^+)$ and $(z_{i,2}^-,z_{i,2}^+)$, and we define $E_i$ as $E_i:=(z_{i,1}^+,z_{i,2}^-)$. We remark that, in general, $\bi<\mathscr{L}(E_i)$, but that $v_x(x)>z_2+\eta$ for every $x\in E_i.$ The energy estimates in Proposition \ref{sotto} below are different for the different types of $D-$intervals given in the following definition (see Figure \ref{ClassPict})
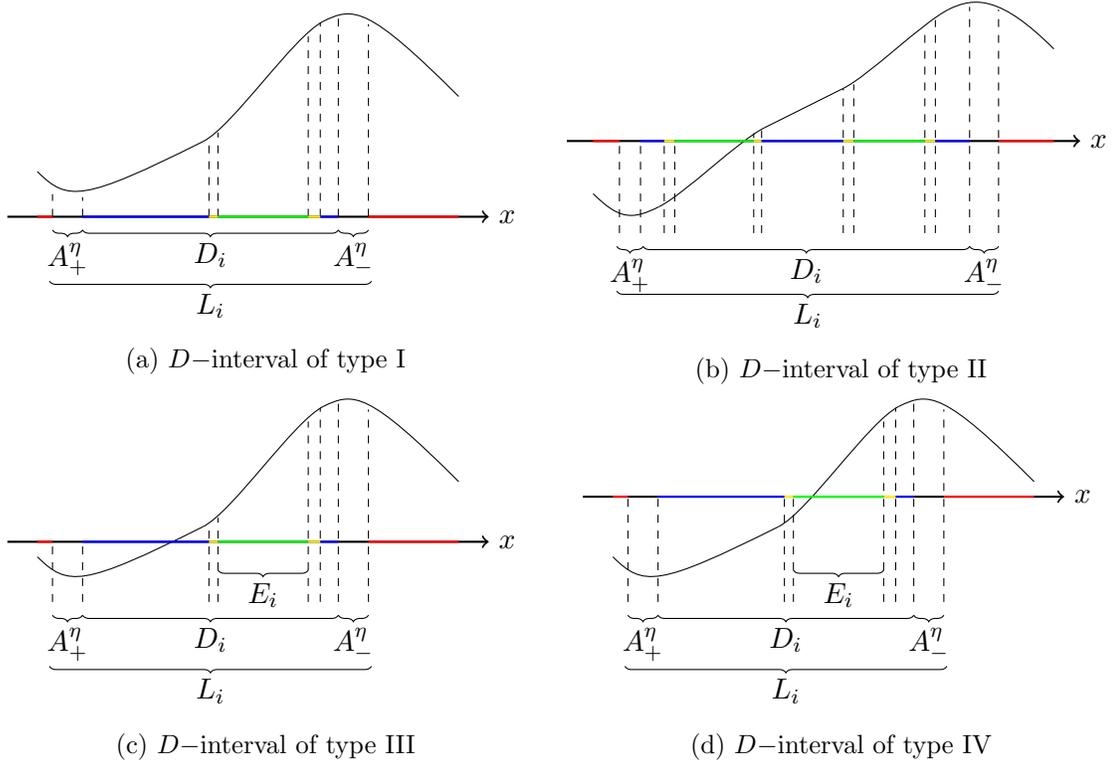
\begin{figure}
\centering
\begin{subfigure}{.45\textwidth}
  \centering
\begin{tikzpicture}[scale=0.4]
\draw (15,-0.5) .. controls (16,-1.5) and (16.5,-1.5) .. (20.5,0.5);
\draw (20.5,0.5) .. controls (21.5,1) and (23.5,4) .. (24.5,4.5);
\draw (24.5,4.5) .. controls (25.5,5) and (26,5) .. (29,2);
\draw [dashed,thin] (15.5,-2) -- (15.5,-1.25);
\draw [dashed,thin] (16.5,-2) -- (16.5,-1.35);
\draw [dashed,thin] (20.7,-2) -- (20.7,0.5);
\draw [dashed,thin] (21,-2) -- (21,0.8);
\draw [dashed,thin] (24.0,-2) -- (24.0,4);
\draw [dashed,thin] (24.4,-2) -- (24.4,4.45);
\draw [dashed,thin] (25,-2) -- (25,4.8);
\draw [dashed,thin] (26,-2) -- (26,4.4);
\draw [->,thick] (14,-2) -- (30,-2);
\filldraw [red] (30,-2) circle (0pt) 
node[anchor=west,black] {$x$};

\draw[decoration={brace,mirror,raise=5pt},decorate]
  (15.5,-2) -- node[below=6pt] {$A_+^\eta$} (16.5,-2);
\draw[decoration={brace,mirror,raise=5pt},decorate]
  (25,-2) -- node[below=6pt] {$A_-^\eta$} (26,-2); 

\draw[decoration={brace,mirror,raise=5pt},decorate]
  (16.5,-2) -- node[below=6pt] {$D_i$} (25,-2);
\draw[decoration={brace,mirror,raise=5pt},decorate]
  (15.4,-3.7) -- node[below=6pt] {$L_i$} (26.1,-3.7);
\draw [red,thick] (15,-2) -- (15.5,-2);
\draw [red,thick] (26,-2) -- (29,-2);
\draw [blue,thick] (16.5,-2) -- (20.7,-2);
\draw [blue,thick] (24.4,-2) -- (25,-2);
\draw [green,thick] (21,-2) -- (24,-2);
\draw [yellow,thick] (20.7,-2) -- (21,-2);
\draw [yellow,thick] (24,-2) -- (24.4,-2);
 \end{tikzpicture}
  \caption{$D-$interval of type I}
  \label{casosemplice I}
\end{subfigure}%
\begin{subfigure}{.45\textwidth}
  \centering
\begin{tikzpicture}[scale=0.35]
\draw (15,-0.5) .. controls (16,-1.5) and (16.5,-1.5) .. (17.5,-1);
\draw (17.5,-1) .. controls (18.5,-0.5) and (20.5,1.5) .. (21.5,2);
\draw (21.5,2) .. controls (22.5,2.5) .. (24.5,3.5);
\draw (24.5,3.5) .. controls (25.5,4) and (27.5,6) .. (28.5,6.5);
\draw (28.5,6.5) .. controls (29.5,7) and (30.5,7) .. (32.5,5);
\draw [dashed,thin] (16,-2) -- (16,1.5);
\draw [dashed,thin] (16.8,-2) -- (16.8,1.5);
\draw [dashed,thin] (17.7,-2) -- (17.7,1.35);
\draw [dashed,thin] (18.1,-2) -- (18.1,1.35);
\draw [dashed,thin] (21.1,-2) -- (21.1,1.8);
\draw [dashed,thin] (21.4,-2) -- (21.4,1.9);
\draw [dashed,thin] (24.5,-2) -- (24.5,3.5);
\draw [dashed,thin] (24.9,-2) -- (24.9,3.65);
\draw [dashed,thin] (27.6,-2) -- (27.6,5.8);
\draw [dashed,thin] (28,-2) -- (28,6.1);
\draw [dashed,thin] (29.3,-2) -- (29.3,6.6);
\draw [dashed,thin] (30.4,-2) -- (30.4,6.5);
\draw [->,thick] (14,1.5) -- (33.5,1.5);
\filldraw [red] (33.5,1.5) circle (0pt) 
node[anchor=west,black] {$x$};

\draw[decoration={brace,mirror,raise=5pt},decorate]
  (15.9,-2) -- node[below=6pt] {$A_+^\eta$} (16.9,-2);
\draw[decoration={brace,mirror,raise=5pt},decorate]
  (29.3,-2) -- node[below=6pt] {$A_-^\eta$} (30.4,-2); 

\draw[decoration={brace,mirror,raise=5pt},decorate]
  (16.9,-2) -- node[below=6pt] {$D_i$} (29.3,-2);
\draw[decoration={brace,mirror,raise=5pt},decorate]
  (15.9,-3.7) -- node[below=6pt] {$L_i$} (30.4,-3.7);
\draw [red,thick] (15,1.5) -- (16,1.5);
\draw [red,thick] (30.4,1.5) -- (32.5,1.5);
\draw [blue,thick] (16.8,1.5) -- (17.7,1.5);
\draw [yellow,thick] (17.7,1.5) -- (18.1,1.5);
\draw [green,thick] (18.1,1.5) -- (21.1,1.5);
\draw [yellow,thick] (21.1,1.5) -- (21.4,1.5);
\draw [blue,thick] (21.4,1.5) -- (24.5,1.5);
\draw [yellow,thick] (24.5,1.5) -- (24.9,1.5);
\draw [green,thick] (24.9,1.5) -- (27.6,1.5);
\draw [yellow,thick] (27.6,1.5) -- (28,1.5);
\draw [blue,thick] (28,1.5) -- (29.3,1.5);
 \end{tikzpicture}
  \caption{$D-$interval of type II}
  \label{casosemplice II}
\end{subfigure}%

\begin{subfigure}{.45\textwidth}
  \centering
\begin{tikzpicture}[scale=0.40]
\draw (15,-0.5) .. controls (16,-1.5) and (16.5,-1.5) .. (20.5,0.5);
\draw (20.5,0.5) .. controls (21.5,1) and (23.5,4) .. (24.5,4.5);
\draw (24.5,4.5) .. controls (25.5,5) and (26,5) .. (29,2);
\draw [dashed,thin] (15.5,-2) -- (15.5,0);
\draw [dashed,thin] (16.5,-2) -- (16.5,0);
\draw [dashed,thin] (20.7,-2) -- (20.7,0.5);
\draw [dashed,thin] (21,-2) -- (21,0.8);
\draw [dashed,thin] (24.0,-2) -- (24.0,4);
\draw [dashed,thin] (24.4,-2) -- (24.4,4.45);
\draw [dashed,thin] (25,-2) -- (25,4.8);
\draw [dashed,thin] (26,-2) -- (26,4.4);
\draw [->,thick] (14,0) -- (30,0);
\filldraw [red] (30,0) circle (0pt) 
node[anchor=west,black] {$x$};

\draw[decoration={brace,mirror,raise=5pt},decorate]
  (15.5,-2) -- node[below=6pt] {$A_+^\eta$} (16.5,-2);
\draw[decoration={brace,mirror,raise=5pt},decorate]
  (25,-2) -- node[below=6pt] {$A_-^\eta$} (26,-2); 
\draw[decoration={brace,mirror,raise=5pt},decorate]
  (21,-0.5) -- node[below=6pt] {$E_i$} (24,-0.5);
\draw[decoration={brace,mirror,raise=5pt},decorate]
  (16.5,-2) -- node[below=6pt] {$D_i$} (25,-2);
\draw[decoration={brace,mirror,raise=5pt},decorate]
  (15.4,-3.7) -- node[below=6pt] {$L_i$} (26.1,-3.7);
\draw [red,thick] (15,0) -- (15.5,0);
\draw [red,thick] (26,0) -- (29,0);
\draw [blue,thick] (16.5,0) -- (20.7,0);
\draw [blue,thick] (24.4,0) -- (25,0);
\draw [green,thick] (21,0) -- (24,0);
\draw [yellow,thick] (20.7,0) -- (21,0);
\draw [yellow,thick] (24,0) -- (24.4,0);
 \end{tikzpicture}
  \caption{$D-$interval of type III}
  \label{casosemplice III}
\end{subfigure}%
\begin{subfigure}{.45\textwidth}
  \centering
\begin{tikzpicture}[scale=0.40]
\draw (15,-0.5) .. controls (16,-1.5) and (16.5,-1.5) .. (20.5,0.5);
\draw (20.5,0.5) .. controls (21.5,1) and (23.5,4) .. (24.5,4.5);
\draw (24.5,4.5) .. controls (25.5,5) and (26,5) .. (29,2);
\draw [dashed,thin] (15.5,-2) -- (15.5,1.5);
\draw [dashed,thin] (16.5,-2) -- (16.5,1.5);
\draw [dashed,thin] (20.7,-2) -- (20.7,1.5);
\draw [dashed,thin] (21,-2) -- (21,1.5);
\draw [dashed,thin] (24.0,-2) -- (24.0,4);
\draw [dashed,thin] (24.4,-2) -- (24.4,4.45);
\draw [dashed,thin] (25,-2) -- (25,4.8);
\draw [dashed,thin] (26,-2) -- (26,4.4);
\draw [->,thick] (14,1.5) -- (30,1.5);
\filldraw [red] (30,1.5) circle (0pt) 
node[anchor=west,black] {$x$};

\draw[decoration={brace,mirror,raise=5pt},decorate]
  (15.5,-2) -- node[below=6pt] {$A_+^\eta$} (16.5,-2);
\draw[decoration={brace,mirror,raise=5pt},decorate]
  (25,-2) -- node[below=6pt] {$A_-^\eta$} (26,-2); 
\draw[decoration={brace,mirror,raise=5pt},decorate]
  (21,-0.5) -- node[below=6pt] {$E_i$} (24,-0.5);

\draw[decoration={brace,mirror,raise=5pt},decorate]
  (16.5,-2) -- node[below=6pt] {$D_i$} (25,-2);
\draw[decoration={brace,mirror,raise=5pt},decorate]
  (15.4,-3.7) -- node[below=6pt] {$L_i$} (26.1,-3.7);
\draw [red,thick] (15,1.5) -- (15.5,1.5);
\draw [red,thick] (26,1.5) -- (29,1.5);
\draw [blue,thick] (16.5,1.5) -- (20.7,1.5);
\draw [blue,thick] (24.4,1.5) -- (25,1.5);
\draw [green,thick] (21,1.5) -- (24,1.5);
\draw [yellow,thick] (20.7,1.5) -- (21,1.5);
\draw [yellow,thick] (24,1.5) -- (24.4,1.5);
 \end{tikzpicture}
  \caption{$D-$interval of type IV}
  \label{casosemplice IV}
\end{subfigure}%
 \caption{\label{ClassPict}Classification of $D-$intervals (cf. Definition \ref{GliDIntervals}). In this picture, the red, the blue and the green intervals are respectively the sets of points where $|v_x-z_1|\leq \eta$, $|v_x-z_2|\leq \eta,$ and $|v_x-z_3|\leq \eta$. The $B^\eta_\pm-$transition layers are coloured in yellow.}
\end{figure}
\begin{dfntn}
\label{GliDIntervals}
Let $v\in V$, $\eta\in(0,\eta_0)$ and $\eps\leq \eta^{q+1}$ be such that $I^\eps(v)\leq C$, with $C>0$. Let $D_i=(x_i^-,x^+_i)$ be the $i-$th $D-$interval for $v$, and $n_i\in 2\mathbb{N}$ be the number of $B_\pm^{\eta}-$transition layers for $v$ in $D_i$. Then we say that $D_i$ is of
\begin{itemize}
\item \textbf{type 0}: if $\max\{\ai,\bi\}\notin(\eps R_*,\eps R^*)$;
\item \textbf{type I}: if $\max\{\ai,\bi\}\in(\eps R_*,\eps R^*)$ and $v(x_i^-)v(x_i^+)\geq 0$;
\item \textbf{type II}: if $\max\{\ai,\bi\}\in(\eps R_*,\eps R^*)$, $v(x_i^-)v(x_i^+)< 0$ and either $n_i=0$ or $n_i\geq 4$;
\item \textbf{type III}: if $\max\{\ai,\bi\}\in(\eps R_*,\eps R^*)$, $v(x_i^-)v(x_i^+)< 0$, $n_i = 2$ and there exists no $x_i^*\in E_i$ such that $v(x_i^*)=0$;
\item \textbf{type IV}: if $\max\{\ai,\bi\}\in(\eps R_*,\eps R^*)$, $v(x_i^-)v(x_i^+)< 0$, $n_i = 2$ and there exists $x_i^*\in E_i$ such that $v(x_i^*)=0$.
\end{itemize}
In this definition $R_*,R^*$ are as in the statement of Lemma \ref{corto}. 
\end{dfntn}

\begin{figure}
\centering
\begin{tikzpicture}[scale=0.40]
\draw (7,4) .. controls (9,5) and (9.5,5) .. (15,-0.5);\draw (15,-0.5) .. controls (16,-1.5) and (16.5,-1.5) .. (20.5,0.5);
\draw (20.5,0.5) .. controls (21.5,1) and (23.5,4) .. (24.5,4.5);
\draw (24.5,4.5) .. controls (25.5,5) and (26,5) .. (28,3);
\draw (28,3) .. controls (29,2) and (29.5,3) .. (33.5,5);
\draw (33.5,5) .. controls (35.5,6) and (36,6) .. (42,0);
\draw [->,thick] (6,1) -- (44,1);
\filldraw [red] (44,1) circle (0pt) 
node[anchor=west,black] {$x$};

\draw [dashed,thin] (15.5,-2) -- (15.5,1);
\draw [dashed,thin] (16.5,-2) -- (16.5,1);
\draw [dashed,thin] (25,-2) -- (25,4.8);
\draw [dashed,thin] (26,-2) -- (26,4.4);
\draw [dashed,thin] (28.2,-2) -- (28.2,2.8);
\draw [dashed,thin] (29.2,-2) -- (29.2,2.7);
\draw [dashed,thin] (34.5,-2) -- (34.5,5.4);
\draw [dashed,thin] (36,-2) -- (36,5.3);

\draw [dashed,thin] (13.5,-2) -- (13.5,1);
\draw [dashed,thin] (41,-2) -- (41,1);

\draw[decoration={brace,mirror,raise=5pt},decorate]
  (12.99,-2) -- node[below=6pt] {$G_i^m$} (15.5,-2);
\draw[decoration={brace,mirror,raise=5pt},decorate]
  (26,-2) -- node[below=6pt] {$G_i^d$} (28.2,-2);\draw[decoration={brace,mirror,raise=5pt},decorate]
  (36,-2) -- node[below=6pt] {$G_i^d$} (41,-2);\draw[decoration={brace,mirror,raise=5pt},decorate]
  (15.5,-2) -- node[below=6pt] {$A_+^\eta$} (16.5,-2);
\draw[decoration={brace,mirror,raise=5pt},decorate]
  (28.2,-2) -- node[below=6pt] {$A_+^\eta$} (29.2,-2);
\draw[decoration={brace,mirror,raise=5pt},decorate]
  (34.5,-2) -- node[below=6pt] {$A_-^\eta$} (36,-2); 
  \draw[decoration={brace,mirror,raise=5pt},decorate]
  (25,-2) -- node[below=6pt] {$A_-^\eta$} (26.1,-2); 
\draw[decoration={brace,mirror,raise=5pt},decorate]
  (16.5,-2) -- node[below=6pt] {$D_i$} (25,-2);
\draw[decoration={brace,mirror,raise=5pt},decorate]
  (28.2,-3.7) -- node[below=6pt] {$L_{i+1}$} (36,-3.7);
\draw[decoration={brace,mirror,raise=5pt},decorate]
  (29.2,-2) -- node[below=6pt] {$D_{i+1}$} (34.5,-2);
\draw[decoration={brace,mirror,raise=5pt},decorate]
  (15.4,-3.7) -- node[below=6pt] {$L_{i}$} (26.1,-3.7);
\draw [red,thick] (9.4,1) -- (15.5,1);
\draw [red,thick] (26,1) -- (28.2,1);
\draw [red,thick] (36,1) -- (42,1);
\draw [blue,thick] (16.5,1) -- (20.7,1);
\draw [blue,thick] (7,1) -- (8.2,1);
\draw [blue,thick] (24.4,1) -- (25,1);
\draw [blue,thick] (29.2,1) -- (34.5,1);
\draw [green,thick] (21,1) -- (24,1);
\draw [yellow,thick] (20.7,1) -- (21,1);
\draw [yellow,thick] (24,1) -- (24.4,1);
 \end{tikzpicture}
 \caption{\label{SigmaIntPict}Examples of sets $\Sigma_i$ constructed in the proof of Proposition \ref{sotto}. Here, $\Sigma_{i+1}=\varnothing$ because $D_{i+1}$ is of type I, while $\Sigma_i = \Sigma_i^d\cup\Sigma_i^m$, where $\Sigma_i^d\subset G_i^d$, $\Sigma_i^m\subset G_i^m.$ The sets $G_i^d,G_i^m$ are constructed as follows: let $D_i=(x_i^-,x_i^+)$, then $G_i^d$ (resp. $G_i^m$) is the intersection of $\{|v_x-z_1|\leq \eta\}$ with the largest interval $(x_i^+,p^+)$, $x_i^+\leq p^+$ (resp. $(p^-,x_i^-)$, $x_i^-\geq p^-$), where $v$ is strictly positive (resp. negative). \\
The red, the blue and the green intervals are respectively the sets of points where $|v_x-z_1|\leq \eta$, $|v_x-z_2|\leq \eta,$ and $|v_x-z_3|\leq \eta$. The $B^\eta_\pm-$transition layers are coloured in yellow.}
\end{figure}
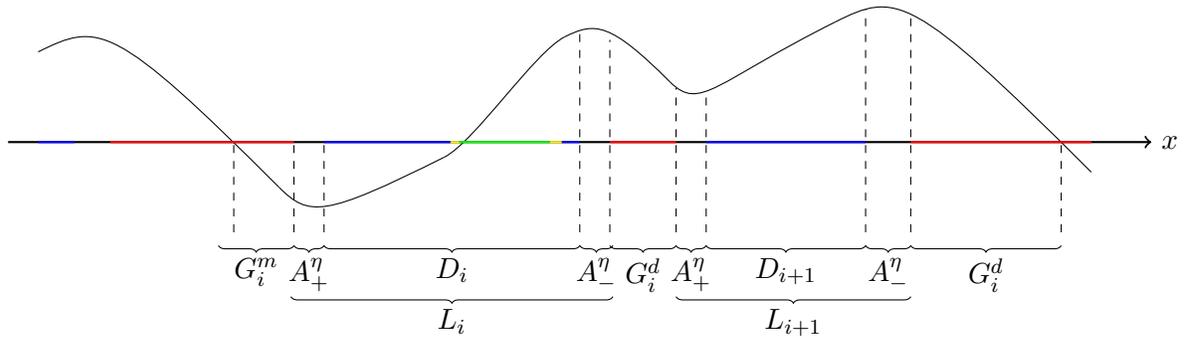

{In Proposition \ref{sotto} below we prove some lower bounds for the $L^2-$norm of $v$ in the $D_i$'s. Then, we identify disjoint sets $\Sigma_i$, $i=1,\dots,N_v$ in which $|v_x-z_1|\leq \eta$, and in these sets we estimate from below the $L^2-$norm of $v$ in terms of $\ai,\bi$. We then combine the estimates in the $\Sigma_i$'s with the estimates in the $L_i$'s, and we argue as in \eqref{minimideps} to obtain a lower bound for the energy on the sets $F_i=L_i\cup\Sigma_i$. The results of Proposition \ref{sotto} are the basic tool to prove both Theorem \ref{bound da sotto thm}, from which follows Theorem \ref{main Thm}, and Theorem \ref{ThmScaling}, from which follows Theorem \ref{main Thm 2}.}

\begin{prpstn}
\label{sotto}
Assume {\rm(H1)--(H5)} and let $C>0$. Then, there exists $\eta_1=\eta_1(C)\in(0,\eta_0)$ such that, for every $\eta\in(0,\eta_1)$, $\eps\leq\eta^{q+1}$ and $v\in V$ satisfying $I^{\eps}(v)\leq C$, 
there exists a collection of (possibly empty) Borel sets $\Sigma_i$, $i=1,\dots,N_v$ such that
$$\Sigma_i\cap\Sigma_j=\Sigma_i\cap L_j=\Sigma_i\cap L_i=\varnothing,\qquad\text{ for every $j\neq i\in\{1,\dots,N_v\}$}, $$
and, for every $i=1,\dots,N_v$,
\beq
\label{CheStimeSotto}
\begin{aligned}
&\int_{L_i\cup\Sigma_i}\bigl(\eps^4v_{xx}^2+\eps^{-2}W(v_x)+\eps^{-2}v^2\bigr)\,\mathrm dx\geq A_0\ai +B_0\bi,\quad&&\text{if $D_i$ is type 0,}\\
&\int_{L_i\cup\Sigma_i}\bigl(\eps^4v_{xx}^2+\eps^{-2}W(v_x)+\eps^{-2}v^2\bigr)\,\mathrm dx\geq \ai f_6^\frac{1}{3}\bigl(\frac{\bi}{\ai}\bigr) - c\eps\eta,\quad&&\text{if $D_i$ is type I, 
}\\
&\int_{L_i\cup\Sigma_i}\bigl(\eps^4v_{xx}^2+\eps^{-2}W(v_x)+\eps^{-2}v^2\bigr)\,\mathrm dx\geq \ai f_7^\frac{1}{3}\bigl(\frac{\bi}{\ai}\bigr) - c\eps\eta,\quad&&\text{if $D_i$ is type II, 
}\\
&\int_{L_i\cup\Sigma_i}\bigl(\eps^4v_{xx}^2+\eps^{-2}W(v_x)+\eps^{-2}v^2\bigr)\,\mathrm dx\geq \ai f_8^\frac{1}{3}\bigl(\frac{\bi}{\ai}\bigr) - c\eps\eta,\quad&&\text{if $D_i$ is type III/IV}
\end{aligned}
\eeq
where $f_6,f_7,f_8$ are as in {\rm(H6)--(H8)}.
\end{prpstn}
\begin{rmrk}
\rm
{The first, the third and the fourth lower bounds in Proposition \ref{sotto} are sharp up to an error proportional to some positive power of $\eps$. Sharpness of the first bound is given by Proposition \ref{upper bnd}. For the third and the fourth bound we refer the reader to the proofs of 
Proposition \ref{H7 prop} and Proposition \ref{H8 prop} respectively.}
\end{rmrk}
\begin{proof}
Let $C>0,$ $v\in V,$ $\eps\leq \eta^{q+1}$ with $I^\eps(v)\leq C$ and $\eta\in(0,\eta_1)$ with $\eta_1$ to be determined later. We divide the proof in three steps: in the first we prove the estimates \eqref{stima 1 sale},\eqref{stima 3 sale},\eqref{stima 2 scende} and \eqref{stima 3 scende} for the $L^2-$norm of $v$ in the $D-$intervals. As explained above, the estimates are different for different types of $D-$intervals. In step two we construct the sets $\Sigma_i$ for every $i=1,\dots,N_v$, and we estimate from below the $L^2-$norm of $v$ on these sets in terms of $\ai,\bi$. Finally, in the last step we combine the estimates for the $L^2-$norm of $v$, with the estimates for the interfacial energy, and deduce \eqref{CheStimeSotto} by means of \eqref{minimideps}.\\

{
\textbf{Step 1:} The strategy to prove estimates for the $L^2-$norm of $v$ in the $D-$intervals is the following: for $i=1,\dots,N_v$ we divide $D_i$ into two intervals (one is actually empty if $D_i$ is of type I), one in which $v$ is bigger or equal than $-c\eps\eta$, one in which $v$ is smaller or equal than $c\eps\eta$. As shown below, for $D-$intervals of type II/IV choosing $x_i^*\in D_i$ such that $v(x_i^*)=0$ and taking $D_i\cap\{x\geq x_i^*\},D_i\cap\{x\leq x_i^*\}$ as sub-intervals suffices. In each interval we approximate $v$ with a suitable continuous piecewise linear function with gradient a.e. in $\{0,z_2,z_3\}$, namely $\sum_{k=2,3}z_k\mathscr{L}\bigl((x_i^*,x)\cap\bigl\{|v_x-z_k|\leq \eta\bigr\}\bigr)$, and use Lemma \ref{dasottointegrale} to bound its $L^2-$norm from below. In conclusion we combine the estimates obtained in the two sub-intervals of $D_i$. The estimates \eqref{stima 2 scende} and \eqref{stima 3 scende} for type III/IV $D-$intervals depend on the quantities $\omega_i^a,\omega_i^b\in[0,1]$ (defined in \eqref{omegas def} below). Estimates \eqref{stima 2 scende} and \eqref{stima 3 scende} need to be combined with the estimates of Step 2 before minimising over $\omega_i^a,\omega_i^b$.}

Let $D_i=(x_i^-,x_i^+)$, and let us first focus on the case $v(x^-_i)v(x^+_i)<0$. We notice that the continuity of $v$, implies the existence of $x_i^*\in(x^-_i,x^+_i)$ such that $v(x_i^*)=0$. We estimate separately $\int_{x_i^*}^{x_i^+}v^2\,\mathrm{d}s$ and $\int_{x_i^-}^{x_i^*}v^2\,\mathrm{d}s$. Let us start with the first. As $v_x>z_1+\eta $ in $(x^-_i,x^+_i)$, we have
\beq
\label{aaa}
v(x) \geq \int_{x_i^*}^x v_x\,\mathrm dx
\geq \sum_{k=2,3}(z_k-\eta)\mathscr{L}\bigl((x_i^*,x)\cap\bigl\{|v_x-z_k|\leq \eta\bigr\}\bigr)+
(z_1+\eta)\mathscr{L}\bigl(\Sigma^\eta\bigr)
\eeq
where $\Sigma^\eta$ is as in \eqref{sigmadelta} and satisfies \eqref{dacitarepoi 000}. 
It follows then from \eqref{aaa} and the assumption $\max\{\ai,\bi\}\leq R^*\eps$ that
\beq
\label{aaa200}
v(x) \geq\sum_{k=2,3}z_k\mathscr{L}\bigl((x_i^*,x)\cap\bigl\{|v_x-z_k|\leq \eta\bigr\}\bigr)-c\frac{\eps^2}{\eta^q}-c\eps\eta R^*,
\eeq
for every $x\in (x_i^*,x_i^+)$. We now use the fact that, as $\eta\in(0,1)$, $(1-\eta)(a+b)^2 \leq a^2 + 2\eta^{-1}b^2 $. This yields
\beq
\label{eqq right}
(1-\eta)\int_{x_i^*}^{x_i^+} \bigl(v(x) +c({\eps^2}\eta^{-q}+\eps\eta)\bigr)^2\mathrm dx 
\leq \int_{x_i^*}^{x_i^+}v^2\,\mathrm dx
+ c \hat r(x_i^+-x_i^*),
\eeq
with $\hat r:=\eta^{-1}({\eps^2}\eta^{-q}+\eps\eta)^2.$
On the other hand, by Lemma \ref{dasottointegrale},
\beq\label{stimona99}
\begin{split}
\int_{x_i^*}^{x_i^+} \bigl(v(x) +c({\eps^2}\eta^{-q}+\eps\eta)\bigr)^2\mathrm dx
&
\geq \int_{x_i^*}^{x_i^+} \Bigl( \sum_{k=2,3}z_k\mathscr{L}\bigl((x_i^*,x)\cap\bigl\{|v_x-z_k|\leq \eta\bigr\}\bigr)\Bigr)^2\mathrm{d}x
\\&\geq z_2^2\int_{0}^{\omega^a_i\ai} x^2\,\mathrm dx+\int_{0}^{\omega^b_i\bi} (z_3x+z_2\omega^a_i\ai)^2 \mathrm dx
\\&\geq 3^{-1}\bigl(z_2^{2}(\omega^a_i\ai)^3+z_3^2(\omega^b_i\bi)^3+g_0 (\ai, \bi , \omega^a_i,\omega^b_i )\bigr)
\end{split}
\eeq
where $\omega^a_i,\omega^b_i\in[0,1]$ are such that
\beq
\label{omegas def} 
\mathscr{L}\bigl(D_i\cap\bigl\{x>x_i^*\colon|v_x-z_2|\leq\eta\bigr\}\bigr)= \omega^{a}_i\ai,\qquad \mathscr{L}\bigl(D_i\cap\bigl\{x>x_i^*\colon|v_x-z_3|\leq\eta\bigr\}\bigr)= \omega^{b}_i\bi,
\eeq 
and $g_0 (a, b , \omega^a,\omega^b )=
3z_2ab\omega^a\omega^b(z_2a\omega^a +z_3b\omega^b ).$ Therefore, putting together \eqref{eqq right}--\eqref{stimona99}, by \eqref{limito intervallo} we obtain
{\beq
\label{stimona} 
\begin{split}
\int_{x_i^*}^{x_i^+}v^2\,\mathrm dx
&\geq (1-\eta)3^{-1}\bigl(z_2^{2}(\omega^a_i\ai)^3+z_3^2(\omega^b_i\bi)^3+g_0 (\ai, \bi , \omega^a_i,\omega^b_i )\bigr) - c\eps^3\eta
\\&\geq 3^{-1}\bigl(z_2^{2}(\omega^a_i\ai)^3+z_3^2(\omega^b_i\bi)^3+g_0 (\ai, \bi , \omega^a_i,\omega^b_i )\bigr) - c\eps^3\eta.
\end{split}
\eeq
Here we also used $\max\{\ai,\bi\}\leq R^*\eps$ and $\eps\leq\eta^{q+1}.$
}In the same way, we prove
\beq
\label{omegab2}
\int_{x_i^-}^{x_i^*}v^2\,\mathrm dx 
+ c\eps^3\eta 
\geq
3^{-1}\bigl(z_2^2((1-\omega^a_i)\ai)^3+z_3^2((1-\omega^b_i)\bi)^3+g_0 (\ai, \bi , 1-\omega^a_i,1-\omega^b_i )\bigr).
\eeq
It turns out that if we sum \eqref{stimona} to \eqref{omegab2} the right hand side is a convex quadratic polynomial in $\omega^a_i,\omega^b_i$ with minimum in $\omega^a_i=\omega^b_i=\frac12$. Therefore, from \eqref{stimona} and \eqref{omegab2} we finally get 
\beq
\label{stima 1 sale}
\int_{D_i} v^2\,\mathrm dx\geq {12}^{-1} h(\ai,\bi) - c\eps^3\eta,\qquad\text{for all $D_i$ of type II,}
\eeq
where
\beq
\label{defHnew}
h(a,b): = z_2^{2} a^3+z_3^2 b^3+3\bigl( a b z_2(z_2 \ai+ z_3\bi)\bigr).
\eeq
The estimate 
\beq
\label{stima 3 sale}
\int_{D_i}v^2\,\mathrm dx\geq 3^{-1} h(\ai,\bi)- c\eps^3\eta,\qquad\text{for all $D_i$ of type I,}
\eeq
follows again from \eqref{stimona} (or \eqref{omegab2}) if $v(x^-_i)>0$ (resp. $v(x^+_i)<0$) which can be deduced via the same argument by setting $x_i^*=x_i^-$ (resp. $x_i^*=x_i^+$). In this case, the above estimates hold with $\omega_i^{a},\omega_i^{b}=1$ (resp. $\omega_i^{a},\omega_i^{b}=0$).\\
Let us suppose now that $v(x_i^-)v(x_i^+)<0$ and that $n_i=2$. {If $D_i$ is of type III, a simple combination of \eqref{stimona}--\eqref{omegab2} leads to
\beq
\label{stima 2 scende}
\int_{D_i}v^2\,\mathrm dx\geq 3^{-1}\hat h_2(\ai,\bi,\omega_i^a,\omega_i^b)- c\eps^3\eta,\qquad\text{for all $D_i$ of type III}, 
\eeq
where
\beq
\label{defHhatnew}
\begin{split}
\hat h_k(a,b,\omega^a,\omega^b)&: = z_2^{2} a^3+z_3^2 b^3+3\bigl(z_2^{2}\omega^a (\omega^a-1)a^3+z_3^{2}\omega^b(\omega^b-1)b^3\bigr)\\ 
&+ z_k3ab\bigl(\omega^a\omega^b(z_2\omega^a a+ z_3\omega^b b)+ (1-\omega^a)(1-\omega^b)(z_2(1-\omega^a)a+ z_3(1-\omega^b) b)\bigr)
\end{split}
\eeq
If $D_i$ is of type IV, we can assume that $x^*_i\in E_i$, and the above estimates can be improved. Indeed, by using Remark \ref{resdasottoint rem} first with $\tau_2=z_3,\tau_1=\tau_0=0,\mathcal{C}_2=\{x\in (x_i^*,x_i^+)\colon |v_x(x)-z_3|\leq \eta\}$, $(a,b) = E_i\cap(x_i^*,x_i^+)$ and then with $(a,b) = (x_i^*,x_i^+)\setminus E_i$ $\tau_2=z_2,\tau_1=0,$ $\tau_0 = z_3\omega_i^b\beta_i$, $\mathcal{C}_2=\{x\in (x_i^*,x_i^+)\colon |v_x(x)-z_2|\leq \eta\}$, we can modify \eqref{stimona99} as follows
\[
\label{stimona 2}
\begin{split}
\int_{{x}_i^*}^{x_i^+} \Bigl( \sum_{k=2,3}z_k\mathscr{L}\bigl((x_i^*,x)\cap\bigl\{|v_x-z_k|\leq \eta\bigr\}\bigr)\Bigr)^2\mathrm{d}x
&\geq z_3^2\int_{0}^{\omega^b_i\bi} x^2\,\mathrm dx+\int_{0}^{\omega^a_i\ai} (z_2x+z_3\omega^b_i\bi)^2 \,\mathrm dx
\\
&\geq
3^{-1}\bigl(z_2^{2}(\omega^a_i\ai)^3+z_3^{2}(\omega^b_i\bi)^3+\frac{z_3}{z_2} g_0 (\ai, \bi , \omega^a_i,\omega^b_i )\bigr),
\end{split}
\]
}which by \eqref{aaa200}--\eqref{eqq right} yields
$$
\int_{x_i^*}^{x_i^+}v^2\,\mathrm dx+ c\eps^3\eta\geq3^{-1}\bigl(z_2^{2}(\omega^a_i\ai)^3+z_3^{2}(\omega^b_i\bi)^3+\frac{z_3}{z_2} g_0 (\ai, \bi , \omega^a_i,\omega^b_i )\bigr).
$$
In the same way, we prove
\beq
\label{omegab2 2}
\begin{split}
\int_{x_i^-}^{x_i^*}v^2\,\mathrm dx&+ c\eps^3\eta\geq 3^{-1}\bigl(z_2^{2}((1-\omega^a_i)\ai)^3+z_3^{2}((1-\omega^b_i)\bi)^3+\frac{z_3}{z_2} g_0 (\ai, \bi ,1- \omega^a_i,1-\omega^b_i )\bigr)
\end{split}
\eeq
By summing up the last two inequalities, we hence get
\beq
\label{stima 3 scende}
\int_{D_i}v^2\,\mathrm dx\geq 3^{-1}\hat h_3(\ai,\bi,\omega_i^a,\omega_i^b)- c\eps^3\eta,\qquad\text{for all $D_i$ of type IV}, 
\eeq
We remark that, in this case, the determinant of the Hessian matrix of $\hat h$ with respect to $\omega_i^a,\omega_i^b$ is negative, and hence $\hat h$ cannot be bounded from below by choosing $\omega_i^a=\omega_i^b=\frac12$.\\

\textbf{Step 2:} For every $i=1,\dots,N_v$, we now construct $\Sigma_i$ and estimate $\int_{\Sigma_i}v^2\,\mathrm{d}x$ from below. Finally, we show that the $\Sigma_i$'s are disjoint. As shown in Figure \ref{SigmaIntPict}, $\Sigma_i=\Sigma_i^d\cup\Sigma_i^m$, where $\Sigma_i^d,\Sigma_i^m$ are subsets of $\{|v_x-z_1|\leq\eta\}$ and respectively of $\{v>0\}$ and $\{v< 0\}.$ 

We first set $\Sigma_i=0$ whenever $D_i$ is of type 0 or of type I. We can hence focus on the $i$'s where $D_i=(x_i^-,x_i^+)$ is such that $v(x_i^+)v(x_i^-)<0$ and $\max\{\ai,\bi\}\in(\eps R_*,\eps R^*)$, that is on type II--IV $D-$intervals. The idea is to bound from below $v$ (or from above) on the set where $|v_x-z_1|\leq\eta$ and $v\geq 0$ (resp. $v\leq 0$) with a continuous piecewise-linear function minus (resp. plus) a small error. We then estimate the $L^2-$norm of the piecewise linear approximation of $v$ and express it in terms of $\ai,\bi$. {We denote by $x^*_i$ a point in $D_i$ (the same that was chosen in Step 1) such that $v(x^*_i)=0$ and such that $x_i^*\in E_i$ if $D_i$ is of type IV. }
From \eqref{aaa200} we have 
\beq
\label{aaa0}
v(x^+_i)\geq z_2 \omega_i^a\ai+z_3 \omega_i^b\bi -c\eta\eps,
\eeq
and, in a similar way, we can prove
\beq
\label{aaa1}
\begin{split}
v(x^-_i)\leq -(z_2 (1-\omega_i^a)\ai + z_3 (1-\omega_i^b)\bi) +c\eta\eps,
\end{split}
\eeq
where $\omega_i^a,\omega_i^b$ are as in \eqref{omegas def}. We now claim that there exist $\eta_1\in(0,\eta_0]$ depending just on $R_*,R^*$ and $C$, such that, for every $\eta<\eta_1$, $v(x_i^+)>0$ and $v(x_i^-)<0$. Indeed, we recall that we are working under the assumption $\max\{\ai,\bi\}\geq R_*\eps$, and we suppose without loss of generality that $\ai\geq R_*\eps$; the case $\bi\geq R_*\eps$ can be treated similarly. Suppose first that $\omega_i ^a\geq \frac12$, then \eqref{aaa0} implies $v(x_i^+)\geq \frac12z_2\eps R_* - c\eta\eps.$ Thus, choosing $\eta$ small enough, we get $v(x_i^+)>0$, and, by the fact that $v(x_i^+)v(x_i^-)<0$, also that $v(x_i^-)<0$. If $\omega_i^a<\frac12$, then $(1-\omega_i^a)>\frac 12$, and \eqref{aaa1} yields $v(x_i^-)\leq -\frac12 z_2 \eps R_* + c\eta\eps,$ so that for every $\eta$ small enough $v(x_i^-)<0$, and hence $v(x_i^+)>0$ as claimed. Now, in the same spirit of \eqref{aaa0}, we have
\begin{align}
\label{aab}
v(x)-v(x_i^+) \geq \int_{x_i^+}^x v_x\,\mathrm dx
 \geq (z_1-\eta)\mathscr{L}\bigl((x_i^+,x)\cap\bigl\{|v_x-z_1|\leq \eta\bigr\}\bigr)-
\tilde r,\qquad x\geq x_i^+,\\
\label{aaa222}
v(x^+_i)\leq \int_{x_i^*}^{x_i^+} v_x \,\mathrm dx\leq \sum_{k=2,3}z_k\mathscr{L}\bigl((x_i^*,x_i^+)\cap\bigl\{|v_x-z_k|\leq \eta\bigr\}\bigr)+\tilde r+c\eta\eps R^*,
\end{align}
where 
\beq
\label{tilder}
\tilde r = \int_{\Sigma^\eta} |v_x|\,\mathrm dx,\qquad	\Sigma^\eta=\bigl\{x\in(0,1)\colon |v_x(x)-z_k|>\eta,\, k=1,2,3	\bigr\}.
\eeq
We now give an estimate for $\tilde{r}$. To this aim, we split $\Sigma^\eta$ into
\[
\Sigma_{1}^\eta := \bigl\{x\in \Sigma^\eta\colon |v_x(x)|\leq t_0\bigr\},\qquad \Sigma_{2}^\eta := \bigl\{x\in \Sigma^\eta\colon t_0<|v_x(x)|\bigr\},
\]
where $t_0$ is such that $W(s)\geq \frac{c_1}{2}|s|^p$ for each $s$ satisfying $|s|>t_0$, and its existence is guaranteed by (H2). By \eqref{dacitarepoi 000}, we have 
\beq
\label{prop 0}
\int_{\Sigma^\eta_{1}} |v_x|\,\mathrm dx\leq t_0 \mathscr{L}(\Sigma^\eta)\leq c \frac{\eps^{2}}{\eta^q}.
\eeq
On the other hand, 
\begin{align*}
\frac{c_1}{2}\mathscr{L}(\Sigma^\eta_2)|t_0|^p \leq \int_{\Sigma^\eta_2} W(v_x)\,\mathrm dx\leq C\eps^2,
\end{align*}
implies $\mathscr{L}(\Sigma^\eta_2)\leq c \eps^2.$ Therefore,
\beq
\label{prop 2}
\int_{\Sigma_2^\eta} |v_x(x)| \,\mathrm{d}x \leq c \int_{\Sigma_2^\eta} W^{\frac1p}(v_x(x)) \,\mathrm{d}x \leq c 
\bigl( \mathscr{L}(\Sigma_2^\eta)\bigr) ^\frac{p-1}{p} \biggl (\int_{\Sigma_2^\eta} W(v_x(x)) \,\mathrm{d}x \biggr )^{\frac1p}
\leq c \eps^{2},
\eeq
where we also made use of the fact that $I^{\eps}(v)\leq C$. 
Collecting \eqref{prop 0}--\eqref{prop 2} we thus get
\beq
\label{bound for tilder}
\tilde{r}\leq \int_{\cup_k\Sigma_k^\eta} |v_x| \,\mathrm dx\leq c(\eps^{2}\eta^{-q}+\eps^{2}) \leq c r^*,
\eeq
with $r^*: = \eps \eta $. By \eqref{aab}--\eqref{aaa222}, \eqref{bound for tilder} together with $\max\{\ai,\bi\}\leq R^*\eps$, we obtain
\beq
\label{aabnuovo}
v(x)-v(x_i^+) 
 \geq (z_1-\eta) \mathscr{L}\bigl((x_i^+,x)\cap\bigl\{|v_x-z_1|\leq \eta\bigr\}\bigr)-
cr^*,\qquad x\geq x_i^+,\eeq
and
\beq
\label{bound su vpiu}
0\leq v(x_i^+) \leq c(\eps\eta+r^*+\eps)\leq c\eps.
\eeq
Now let 
$$ 
g_b(x)= v(x_i^+) - cr^* - (|z_1|+\eta)\mathscr{L}\bigl((x_i^+,x)\cap\bigl\{|v_x-z_1|\leq \eta\bigr\}\bigr).
$$
Let $\bar x_i\in[x^+_i,1]$ be the smallest $x$ in $[x^+_i,1]$ such that $g_b(x)\leq 0$, and let us set
$$
\Sigma_i^d:=(x_i^+,\bar x_i)\cap\bigl\{|v_x-z_1|\leq \eta\bigr\}.
$$
The existence of $\bar{x}_i$ is guaranteed by the continuity of $v$ and the fact that $v(1)=0$ together with \eqref{aabnuovo} imply $g_b(1)\leq 0$. Thus, $\Sigma_i^d=\varnothing$ if and only if $g_b(x_i^+)\leq0$. Now, if $\Sigma_i^d\neq\varnothing$, that is $v(x_i^+)> cr^*$, by \eqref{aabnuovo} we have that 
\beq
\label{positivoV}
v(x)\geq g_b(x)> 0,\qquad \text{for ever $x\in (x_i^+,\bar{x}_i)$}.
\eeq 
Now, from \eqref{positivoV} we deduce that
\[
\begin{split}
\int_{\Sigma_i^d} v^2\,\mathrm dx &\geq \int_{\Sigma_i^d} g_b^2(x)\,\mathrm dx 
\geq (|z_1|+\eta)^{-1}\int_{0}^{v(x^+_i)-r^*}x^2\,\mathrm dx\geq (|z_1|+\eta)^{-1} v^3(x_i^+)-c\eps^3\eta.
\end{split}
\]
Here we have used a change of variable $y=g_b(x)$, and, in the last inequality, we exploited \eqref{bound su vpiu}.  {The same lower bound holds trivially if $\Sigma_i^d=\varnothing$, and hence $v(x_i^+)\leq cr^*$.}
Therefore, as $(|z_1|+\eta)^{-1}\geq |z_1|^{-1}-|z_1|^{-2}\eta$, by \eqref{aaa0} we obtain
\beq
\label{stima 1 a}
\int_{\Sigma_i^d} v^2 \,\mathrm dx\geq 3^{-1}|z_1|^{-1}\bigl(z_2\omega_i^a\ai +\omega_i^b\bi z_3\bigr)^3 - cr_b,
\eeq
where $r_b:=\eps^3 \eta$ and we made use of \eqref{bound su vpiu} and the fact that $\max\{\ai,\bi\}\leq \eps R^*$. In the same way, letting  
$$
g_m(x) = v(x_i^-) + cr^* + (|z_1|+\eta)\mathscr{L}\bigl((x,x_i^-)\cap\bigl\{|v_x-z_1|\leq \eta\bigr\}\bigr),
$$
and $\tilde x_i\in[0,x_i^-]$ be the largest $x\in[0,x_i^-]$ such that $g_m(x)\geq 0$, we can define
$$
\Sigma_i^m:=(\tilde x_i,x_i^-)\cap\bigl\{|v_x-z_1|\leq \eta\bigr\}
$$
and deduce
\beq
\label{stima 1 b}
\int_{\Sigma_i^m} v^2\,\mathrm dx \geq 3^{-1}|z_1|^{-1}\bigl(z_2(1-\omega_i^a)\ai +(1-\omega_i^b)\bi z_3\bigr)^3 - cr_b.
\eeq
Define now $\Sigma_i:=\Sigma_i^m\cup\Sigma_i^d$, then by \eqref{stima 1 a}--\eqref{stima 1 b} we deduce
\beq
\label{stima 1 general}
\int_{\Sigma_i} v^2 \,\mathrm dx + cr_b\geq 3^{-1}|z_1|^{-1}\bigl(z_2\omega_i^a\ai +\omega_i^b\bi z_3\bigr)^3 +3^{-1}|z_1|^{-1}\bigl(z_2(1-\omega_i^a)\ai +(1-\omega_i^b)\bi z_3\bigr)^3 .
\eeq
We remark that $v_x(x)\leq z_1+\eta$ for each $x\in\Sigma_i,$ while $v_x(x)> z_1+\eta$ in every $L_j$, $j=1,\dots N_v$. In this way $L_j\cap\Sigma_i  =\varnothing$ for every $i,j=1,\dots N_v$. We now claim that $\Sigma_j\cap\Sigma_i=\varnothing$ for every $j=1,\dots N_v$ with $i\neq j$. Indeed, let $x^-_j,x_j^+$ be such that $D_j=(x^-_j,x_j^+)$, then in case $v(x_j^-)v(x^-_j)\geq0$ we defined $\Sigma_j=\varnothing$ and the conclusion follows trivially. We can hence focus on the case $v(x^-_j)v(x_j^+)<0$, and recall that $\eta<\eta_1$ implies $v(x_j^+)>0$ and $v(x_j^-)<0.$
Now, the construction of the $\Sigma_j^d,\Sigma_j^m$ is such that $\Sigma_j^d\subset(x_j^+,\bar x_j)$ (resp. $\Sigma_j^m\subset(\tilde x_j, x_j^-)$). Furthermore, as stated in \eqref{positivoV}, $v$ is strictly positive in $(x_i^+,\bar x_i)$ (resp. strictly negative in $(\tilde x_i, x_i^-)$). Therefore, $\Sigma_i^d\cap \Sigma_j^m=\varnothing$ for every $i,j=1,\dots,N_v$. Finally, assuming without loss of generality $i<j$, we have that $\Sigma_i^d\subset (x_i^+,\bar x_i)$ and $\Sigma_j^d\subset (x_j^+,\bar x_j)$ with $v(x)>0$ for every $x\in(x_i^+,\bar x_i)\cup(x_j^+,\bar x_j)$. But as $v(x_j^-)<0$ and $x_j^-\in (x_i^+,x_j^+)$, $(x_i^+,\bar x_i)$ and $(x_j^+,\bar x_j)$ have to be disjoint, and so must be $\Sigma_i^d$ and $\Sigma_j^d$. In the same way we prove $\Sigma_i^m\cap \Sigma_j^m=\varnothing$ for every $j=1,\dots,N_v$, $i\neq j$, concluding the proof of the claim.\\

\textbf{Step 3:} We now combine the estimates of Step 1 with the estimates of Step 2, and use an argument as the one in \eqref{minimideps} to deduce the bounds in \eqref{CheStimeSotto} for type I--IV $D-$intervals. In the case of type II/IV $D-$intervals a minimisation over $\omega_i^a,\omega_i^b$ is performed to get lower bounds independent of these parameters.

{We start by noticing that Lemma \ref{corto} leads to \eqref{CheStimeSotto} in the case of type 0 $D-$intervals. Now, we notice that, by \eqref{MMest}, 
\beq
\label{stima sopra per sempre}
\begin{split}
\int_{L_i\cup\Sigma_i}\bigl(\eps^4v_{xx}^2+\eps^{-2}W(v_x)\bigr)\,\mathrm dx &\geq \int_{L_i}\bigl(\eps^4v_{xx}^2+\eps^{-2}W(v_x)\bigr)\,\mathrm dx\\
\geq 2\eps\bigl(H(z_2-\eta) &-H(z_1+\eta)\bigr)+n_i\eps\bigl(H(z_3-\eta) -H(z_2+\eta)\bigr)\\
&\geq \eps(2E_0+n_iE_1)-c\eta\eps,
\end{split}
\eeq
for every $i=1,\dots,N_v$. For the $i$'s where $D_i$ is of type I, we set $\Sigma_i=\varnothing$, and thus, by \eqref{stima 3 sale},
\beq
\label{boundsotto01dacit}
\int_{L_i\cup\Sigma_i}\eps^{-2}v^2\,\mathrm dx+c\eps\eta 
\geq 3^{-1}\eps^{-2}h(\ai,\bi),\qquad\text{if $D_i$ is type I.}
\eeq
If $n_i=0$ or $n_i\geq4$, we can minimise the right hand side of \eqref{stima 1 general} over $\omega_i^a,\omega^b_i\in[0,1]$. This is a convex quadratic function attaining its minimum at $\omega_i^a=\omega^b_i=\frac12$. Thus, by \eqref{stima 1 sale} we get
\beq
\label{idontknowwhat}
\int_{L_i\cup\Sigma_i}\eps^{-2}v^2\,\mathrm dx+c\eps\eta 
\geq \frac1{12\eps^2}\Bigl(z_2^2z_{21}\ai^3+z_3^2z_{31}\bi^3+3\ai\bi z_2z_{31}(\bi z_3+\ai z_2) \Bigr),\qquad\text{if $D_i$ is type II.}
\eeq
In case of type III $D-$intervals, we recall that $\omega_i^b\in\{0,1\}$. Therefore, {we assume without loss of generality that $\omega_i^b=1$ (the case $\omega_i^b=0$ can be treated similarly) and by \eqref{stima 1 general} we deduce}
\beq
\label{se devo proprio}
\begin{split}
cr_b+\int_{\Sigma_i\cup D_i } 
v^2\,\mathrm dx   &\geq 3^{-1}\Bigl( \hat{h}_2(\ai,\bi,\omega_i^a,1)+|z_1|^{-1}\bigl(z_2\ai\omega_i^a+z_3\bi\bigr)^3+|z_1|^{-1}\bigl(z_2\ai(1-\omega_i^a)\bigr)^3\Bigr)
\\
&
\geq 
3^{-1}\Bigl( z_2^2z_{21}\ai^3+z_3^2z_{31}\bi^3 -3\ai^3 f_0\Bigl(\frac{\bi}{\ai}\Bigr)\Bigr),
\end{split}
\eeq
where, $f_0$ is defined by 
\beq
\label{deff0}
f_0(y)=\frac{(y^2 z_{31}z_3-z_2z_{21})^2}{4(z_{21}+yz_{31})}.
\eeq
We remark that the last lower bound in \eqref{se devo proprio} is sharp if and only if $\frac{\bi}{\ai}\leq \sqrt{\frac{z_2z_{21}}{z_3z_{31}}}$. For type IV $D-$intervals, \eqref{stima 1 general}, together with \eqref{stima 3 scende} yield
\beq
\label{ultimaspero}
\begin{split}
cr_b+\int_{\Sigma_i\cup D_i } 
v^2 \,\mathrm dx  &\geq h^*(\ai,\bi,\omega_i^a,\omega_i^b)
\end{split}
\eeq
where,
\[
\begin{split}
h^*(\ai,\bi,\omega_i^a,\omega_i^b) := 3^{-1}|z_1|^{-1}&\bigl(z_2\ai(1-\omega_i^a)+z_3\bi(1-\omega_i^b)\bigr)^3 \\
&+ 3^{-1}\Bigl( \hat h_3(\ai,\bi,\omega_i^a,\omega_i^b)+|z_1|^{-1}\bigl(z_2\ai\omega_i^a+z_3\bi\omega_i^b\bigr)^3	\Bigr).
\end{split}
\]
We claim, that
\beq
\label{minimiIVh}
\min_{(\omega^a,\omega^b)\in[0,1]^2} h^*(\ai,\bi,\omega^a,\omega^b)
\geq 
3^{-1}\Bigl( z_2^2z_{21}\ai^3+z_3^2z_{31}\bi^3 -3\ai^3 f_0\Bigl(\frac{\bi}{\ai}\Bigr)\Bigr),
\eeq
with $f_0$ is as in \eqref{deff0}. Indeed, $h^*$ is a second order polynomial in $\omega_i^a,\omega_i^b$ with negative Hessian determinant. Therefore, the minimum among the $\omega_i^a,\omega_i^b\in[0,1]$ is attained at $\omega_i^a\in\{0,1\}$, or $\omega_i^b\in\{0,1\}$. More precisely, if $\frac\bi\ai\geq \sqrt{\frac{z_2z_{21}}{z_3z_{31}}}$, the minimum is attained at $\omega_i^a\in\{0,1\}$ and a minimization over $\omega_i^b\in [0,1]$ gives \eqref{minimiIVh}. The same lower bound can be achieved when $\frac\bi\ai< \sqrt{\frac{z_2z_{21}}{z_3z_{31}}}$. Indeed, in this case, the minimum is attained at $\omega_i^b\in\{0,1\}$, and, by using the fact that $z_2<z_3$, we can bound from below $h^*(\ai,\bi,\omega_i^a,\omega^b_i)$ with 
$$
3^{-1}\Bigl(\hat h_2(\ai,\bi,\omega_i^a,\omega_i^b)+|z_1|^{-1}\bigl(z_2\ai\omega_i^a+z_3\bi\bigr)^3+|z_1|^{-1}\bigl(z_2\ai(1-\omega_i^a)\bigr)^3\Bigr),
$$
which by \eqref{se devo proprio} yields again to \eqref{minimiIVh}. Therefore, for type III/IV $D-$intervals \eqref{se devo proprio}-\eqref{minimiIVh} imply
\beq
\label{normaL2IIIeIV}
\int_{L_i\cup\Sigma_i}\eps^{-2}v^2\,\mathrm dx+c\eps\eta 
\geq 3^{-1}\Bigl( z_2^2z_{21}\ai^3+z_3^2z_{31}\bi^3 -3\ai^3 f_0\Bigl(\frac{\bi}{\ai}\Bigr)\Bigr),\qquad\text{if $D_i$ is type III/IV.}
\eeq
Finally, by combining \eqref{stima sopra per sempre}, \eqref{boundsotto01dacit}, \eqref{idontknowwhat} and \eqref{normaL2IIIeIV}, and by arguing as in \eqref{minimideps}, we obtain \eqref{CheStimeSotto}.}
\end{proof}

As a corollary of the previous result, we can prove
\begin{thrm}
\label{bound da sotto thm}
Assume {\rm(H1)--(H8)}, and let $C>0$. Then there exists $\eta_1=\eta_1(C)\in(0,\eta_0)$ such that, if $\eta\in(0,\eta_1)$, $\eps\leq\eta^{q+1}$ and $v\in V$ satisfies $I^{\eps}(v)\leq C$, it holds
\beq
\label{tesi inf 000}
\Mint\bigl(\eps^4 v_{xx}^2+\eps^{-2}W(v_x)+\eps^{-2}v^2\bigr)\,\mathrm dx\geq \bigl( A_0\lambda_2^\eta+B_0\lambda_3^\eta\bigr) - c\eta.
\eeq
\end{thrm}

\begin{proof}
Thanks to Proposition \ref{sotto} and (H6)--(H8) we have
$$
\int_{ L_i\cup\Sigma_i} \bigl(\eps^4v_{xx}^2+\eps^{-2}W(v_x)+\eps^{-2}v^2\bigr)\,\mathrm dx\geq  A_0\ai+B_0\bi - c\eta \eps,
$$
for every $i=1,\dots,N_v$. It just remains to provide an estimate for the intervals $D_0:=(0,y_1^-)$ and $D_{{N_v}+1}:=(y_{N_v}^+,1)$. We deal with the first case, as the second can be treated similarly. If $\mathscr{L}(\{x\in D_0\colon |v_x(x)-z_k|\leq\eta\})>R^*\eps$ for some $k=2,3$, {then, by arguing as in the proof of Lemma \ref{corto} (cf. \eqref{circle}--\eqref{chestima}) we deduce}
$$
\int_{D_0}\bigl(\eps^4v_{xx}^2+\eps^{-2}W(v_x)+\eps^{-2}v^2\bigr)\,\mathrm dx\geq A_0\alpha_0+B_0\beta_0.
$$
On the other hand, if $\mathscr{L}(\{x\in D_0\colon |v_x(x)-z_k|\leq\eta\})\leq R^*\eps$ for $k=1,2$, then
\beq
\label{D0 DN}
\int_{D_0}\bigl(\eps^4v_{xx}^2+\eps^{-2}W(v_x)+\eps^{-2}v^2\bigr)\,\mathrm dx\geq A_0\alpha_0+B_0\beta_0 - (A_0+B_0)R^*\eps.
\eeq
Therefore, recalling that $N_v\leq c\eps^{-1}$ (see Lemma \ref{numero})
$$
\Mint\bigl(\eps^4v_{xx}^2+\eps^{-2}W(v_x)+\eps^{-2}v^2\bigr)\,\mathrm dx\geq  A_0\sum_{i=0}^{{N_v}+1}\ai+B_0\sum_{i=0}^{{N_v}+1}\bi - c\eta,
$$
which, by the definition of the $\ai$'s, $\bi$'s, and of $\lambda^\eta_2,\lambda_3^\eta$ (see \eqref{sommediaibi}) coincides with \eqref{tesi inf 000}.
\end{proof}

\section{The second $\Gamma$--limit}
\label{section 5}
In this section we prove the $\Gamma-$limit for $I^\eps$, that is a second $\Gamma-$limit for $\EE^\eps$, as stated in Theorem \ref{main Thm}. The first step is to prove compactness for the family of energy functionals $I^\eps$.
\begin{prpstn}
\label{compattezza}
Assume {\rm(H1)--(H5)}. Let $C>0$, $\eps_j\downarrow0$ and $(\uj,\nu_j)\in L^2(0,1)\times L^\infty_{w^*}(0,1;\mathcal{M})$ be such that 
\beq
\label{comp 0}
I^{\ej}(\uj,\nu_j) \leq C,\qquad \forall j \in \mathbb{{N}}.
\eeq
Then, up to a subsequence, $(\uj,\nu_j)$ converges to $(u,\nu)$ in $L^2(0,1)\times L^\infty_{w^*}(0,1;\mathcal{M})$. Furthermore, $u=0$ and $\nu\in \mathrm{GYM}^\infty(u)$ satisfies
\beq
\label{comp 02}
 \nu_x = \lambda_1(x)\delta_{z_1}+\lambda_{2}(x)\delta_{z_2}+\lambda_{3}(x)\delta_{z_3},\qquad \text{ a.e. $x\in(0,1)$},
\eeq
where $\lambda_1,\lambda_{2},\lambda_{3}\in L^\infty(0,1;[0,1])$ are such that
\beq
\label{comp 01}
\lambda_{1}(x)+\lambda_{2}(x)+\lambda_3(x)=1,\qquad\text{and}\qquad \sum_{k=1}^3z_k\lambda_k(x)=0,
\eeq
for a.e. $x\in(0,1)$.
\end{prpstn}
\begin{proof}
We first notice that \eqref{comp 0} implies strong convergence of $\uj$ to $u=0$ in $L^2(0,1)$. Furthermore, as $W(s) \geq c_1 |s|^p-c_2$, we also have
\beq
\label{comp 1}
\|\ux\|_{L^p}\leq c,
\eeq
and, therefore, up to a subsequence $u_j \to 0$, weakly in $W^{1,p}_0(0,1)$. In fact, \eqref{comp 1} also implies that, up to a further non-relabelled subsequence, 
$\ux$ generates a gradient Young measure $\nu_x$, weak$*$ limit of $\nu_j$ in $L^\infty_{w^*}(0,1;\mathcal{M})$. Defined $\Sigma_j^\eta$ as
\beq
\label{nuovoSigma}
\Sigma_j^\eta:=\bigl\{x\in(0,1)\colon |\ux(x)-z_k|>\eta,\, k=1,2,3	\bigr\},
\eeq
for some $\eta\in (0,\eta_0)$, by (H5) we have
$$
C\eps_j^2\geq\uint W(\ux)\,\mathrm dx\geq\int_{\Sigma_j^\alpha} W(\ux)\,\mathrm dx\geq c_0\eta^q\mathscr{L}(\Sigma_j^\eta).
$$
This implies
\beq
\label{comp 2}
\mathscr{L}(\Sigma_j^\eta) \leq c\frac{\eps_j^2}{\eta^q},
\eeq
which is convergence in measure of $\ux$ to $\mathcal{Z}$. Therefore, $\nu_x$ is a probability measure supported on $\mathcal{Z}$ (see e.g., \cite{BallYM}), and hence $\nu_x = \lambda_{1}(x)\delta_{z_1}+\lambda_2(x)\delta_{z_2}+\lambda_{3}(x)\delta_{z_3}$ for a.e. $x\in(0,1)$, as claimed. The fact that $\nu$ is a probability measure implies the first identity in \eqref{comp 01}. By \cite[Thm. 8.7]{Pedregal} we also know that $\nu$ is the gradient Young measure related to $u=0$, and therefore the average of $\nu$ must be $0$, that is $\sum_{k=1}^3\lambda_k(x)z_k =0$ for a.e. $x\in(0,1)$, which is the last identity in \eqref{comp 01}. 
\end{proof}

Given a sequence $u_j\in  W^{1,p}_0(0,1)$ and $\eta>0$, let us define
$$
\lambda_{k,j}^{\eta}:=  \mathscr{L}\bigl(\bigl\{x\in (0,1)\colon |u_{j,x}(x)-z_k|\leq\eta\bigr\}\bigr),\qquad k=1,2,3.
$$
The following result is used below:
\begin{lmm}
\label{limite la}
Assume {\rm(H1)--(H5)}. Let $C>0$, $\eta\in(0,\eta_0)$, $\eps_j\downarrow0$ and $(\uj,\nu_j)\in L^2(0,1)\times L^\infty_{w^*}(0,1;\mathcal{M})$ be a sequence converging to $(0,\nu)$ in $L^2(0,1)\times L_{w^*}^\infty(0,1;\mathcal{M})$ such that $I^{\ej}(\uj)\leq C$ for each $j$. Then $\nu$ satisfies \eqref{comp 02}--\eqref{comp 01} and 
\beq
\label{1limitedil}
\lim_j \lambda_{k,j}^{\eta} = \Mint \lambda_k\,\mathrm dx.
\eeq
\end{lmm}
\begin{proof}
The fact that $\nu$ satisfies \eqref{comp 02}--\eqref{comp 01} follows directly from Proposition \ref{compattezza}. We just need to prove \eqref{1limitedil}. Let us consider a continuous function $f_k\colon \R\to[0,1]$, which is equal to $1$ for those $s$ such that $|s-z_k|\leq\eta$, and equal to $0$ for $|s-z_k|\geq \eta_0$. We have
$$
\uint\lambda_k\,\mathrm dx = \uint\langle \nu , f_k \rangle \,\mathrm dx= \lim_j\uint\langle \nu_j,f_k\rangle \,\mathrm dx= \lim_j\uint f_k(\ux)\,\mathrm dx.
$$
Now, we notice that, as $\eta_0<\frac{|z_k-z_h|}{2}$ for each $h\neq k\in\{1,2,3\}$ (cf. (H5)),
$$
\uint f_k(\ux) \,\mathrm dx= \lambda_{k}^{\eta}(u_j) + r,
$$
{where $
0<r \leq \mathscr{L}\bigl( \Sigma_j^\eta\bigr)\leq c \ej^2\eta^{-q}.$ Here, $\Sigma_j^\eta$ is as in \eqref{nuovoSigma}, and was estimated by means of \eqref{comp 2}.} Therefore, collecting all previous identities we finally get
$$
\lim_j\lambda_k^{\eta}(u_j) = \uint\lambda_k\,\mathrm dx,
$$
which concludes the proof.
\end{proof}

\subsection{Proof of Theorem \ref{main Thm}}
{By \cite[Remark 1.29]{Braides} we just need to show the $\Gamma-\limsup$ inequality for every $\nu\in X$, where $X$ is the set containing all $\nu\in \mathrm{GYM}^\infty(0)$, with $\supp\nu\subset\mathcal Z,$ and such that $\lambda_2:=\nu(z_2) \colon(0,1)\to\{0,z_{21}^{-1}\}$ is constant on every sub-interval $(x_i, x_{i+1})$, $i=1,\dots,n-1$, for some partition $0= x_1\leq\dots\leq x_n = 1$ of $(0,1)$ and some $n\in\mathbb{N}$. Indeed $X$ is dense with respect to the weak$*$ topology of $L^\infty_{w^*}(0,1;\mathcal{M})$ in the set containing all $\nu\in\mathrm{GYM}^\infty(0)$ such that $\supp\nu\subset\mathcal Z.$ This is because the space of piecewise constant functions in $L^{\infty}(0,1;\{0,z_{21}^{-1}\})$ is weak$*$ dense in the space of piecewise constant functions in $L^\infty(0,1;(0,z_{21}^{-1}))$ (cf. \cite[Pb. 1, Sec. 8.6]{Evans}), which is weakly$*$ dense in $L^{\infty}(0,1;(0,z_{21}^{-1}))$. On $X$ the $\Gamma-\limsup$ follows directly by Proposition \ref{upper bnd} and Lemma \ref{limite la}.} Therefore we just need to prove the $\Gamma-\liminf$ inequality. In order to do that, we need to consider a generic sequence $\ej$ converging to $0$, a sequence $\uj\in V$ converging strongly in $L^2(0,1)$ to $u\in L^2(0,1)$ and a sequence of parametrized measures $\nu_j\in L^\infty_{w^*}(0,1;\mathcal M)$ converging weakly$*$ to $\nu$ in $L^\infty_{w^*}(0,1;\mathcal M)$. If $\liminf_j I^{\ej}(\uj) =\infty$, the liminf inequality is trivial. Otherwise, up to a subsequence we can assume the existence of $C>0$, independent of $\eps_j$, such that
$$
I^{\ej}(\uj)\leq C,\qquad \forall j\in\mathbb{{N}}.
$$
In this case, Proposition \ref{compattezza} implies that $u=0$ and that $\nu\in \mathrm{GYM}^\infty(0)$ satisfies \eqref{comp 02}--\eqref{comp 01}. 
Now, Theorem \ref{bound da sotto thm} guarantees the existence of $\eta_1>0$ such that, fixed $\eta\in (0,\eta_1)$,
$$
\Mint\bigl(\eps_j^4 u_{j,xx}^2+\eps_j^{-2}W(u_{j,x})+\eps_j^{-2}\uj^2\bigr)\,\mathrm dx\geq \bigl( A_0\lambda_{2,j}^\eta+B_0\lambda_{3,j}^\eta\bigr) - c\eta,
$$
for all $\eps_j<\eta^{q+1}$. Now, by taking the $\liminf$ on both sides, and recalling Lemma \ref{limite la}, we deduce 
$$
\liminf_j\Mint\bigl(\eps_j^4 u_{j,xx}^2+\eps_j^{-2}W(u_{j,x})+\eps_j^{-2}\uj^2\bigr)\,\mathrm dx\geq A_0 \Mint\lambda_2\,\mathrm dx+B_0\Mint\lambda_3\,\mathrm dx - c\eta.
$$
The arbitrariness of $\eta$ yields to the desired $\Gamma-\liminf$ inequality.

\section{Selecting Minimizing sequences without $\Gamma-$convergence}
\label{Che bound}
In this section we prove Theorem \ref{main Thm 2}. In order to do this we strongly rely on the estimates of Section \ref{section 4}, to which we refer the reader for the notation. We start with the following theorem: 
\begin{thrm}
\label{ThmScaling}
Assume {\rm(H1)--(H5)} and $z_3\leq 3|z_1|$. Then, there exist $\eps_1>0$ and $\xi>0$ such that, if 
$\eps<\eps_1$ 
\beq
\label{tesi inf}
\inf_{v\in V}\Mint\bigl(\eps^4 v_{xx}^2+\eps^{-2}W(v_x)+\eps^{-2}v^2\bigr)\,\mathrm dx= A_0 z_{21}^{-1} + o(\eps^\xi).
\eeq
Furthermore, every minimizer $u$ of $I^\eps$ satisfies
\beq
\label{tesi inf 2}
\mathscr{L}\bigl(\{x\in(0,1)\colon |u_x(x)-z_3|\leq \eps^\frac1{q+1} \}\bigr) \leq c\eps^\xi.
\eeq
\end{thrm}
\begin{proof}
We first notice that Proposition \ref{upper bnd} implies 
\beq
\label{da sopra ok}
\inf_{v\in V} \Mint\bigl(\eps^4 v_{xx}^2+\eps^{-2}W(v_x)+\eps^{-2}v^2\bigr)\,\mathrm dx\leq A_0 z_{21}^{-1} + c\eps^\zeta.
\eeq
Let us assume $\eta\in (0,\eta_1)$, with $\eta_1$ as in the statement of Proposition \ref{sotto}, $\eps\leq \eta^{q+1}$, and define
$$K:=\frac{z_{31}}{z_{21}}\sqrt[3]{\frac{(E_0+E_1)^2}{E_0^2}}>1.$$ 
We notice that $B_0>KA_0$. We now look for a lower bound for the energy $I^\eps$ of the type \eqref{tesi inf 000}, but with $B_0$ replaced by $KA_0$. This new bound does not rely on (H6)--(H8), but is deduced by strongly exploiting the estimates in Section \ref{section 4}. \\

It can be checked by using $z_2<z_3$, $z_{21}<z_{31}$, $z_2z_{31}<z_3z_{31},$ that $f_7,f_8$ given in (H7)--(H8) satisfy
\beq
\label{sottof7}
f_7^{\frac13}(y)\geq A_0 + KA_0 y,\qquad f_8^{\frac13}(y)\geq A_0 + KA_0 y,\qquad\text{for every $y\geq 0$}.
\eeq
Furthermore, as we assumed $z_3\leq 3|z_1|$, we also have
\beq
\label{sottof6}
f_6^{\frac13}(y)\geq A_0 + KA_0 y,\qquad\text{for every $y\geq 0$}.
\eeq
Therefore, collecting the estimates for every $L-$interval, from Proposition \ref{sotto} together with \eqref{sottof7}--\eqref{sottof6} and the fact that $B_0>KA_0$ we deduce
\[
\int_{\sum_i(L_i\cup\Sigma_i)}\bigl(\eps^4 v_{xx}^2 +\eps^{-2}W(v_x)+\eps^{-2}v^2\bigr)\,\mathrm dx+ c\eta
\geq
A_0\sum_i(\ai+K\bi).
\]
Here we also made use of Lemma \eqref{numero} to bound ${N_v}$ with $c\eps^{-1}$. Finally, recalling \eqref{D0 DN} we deduce
\beq
\label{stimettada2soldi}
\Mint \bigl(\eps^4 v_{xx}^2 +\eps^{-2}W(v_x)+\eps^{-2}v^2\bigr)\,\mathrm dx+ c\eta
\geq
A_0 (\lambda_2^\eta+K\lambda_3^\eta).
\eeq
Now, thanks to \eqref{sigmadelta}--\eqref{dacitarepoi 000}, and the fact that $\eps\leq \eta^{q+1}$, we can write
\beq
\label{ineq 001}
1\geq \lambda_1^\eta+\lambda_2^\eta+\lambda_3^\eta=1-\mathscr{L}(\Sigma^\eta)\geq 1 - c\eps\eta,
\eeq
while, on the other hand,
\beq
\label{ineq 002}
0 = \Mint v_x \,\mathrm dx\leq \sum_{k=1}^3(z_k+\eta)\lambda_k^\eta+\tilde{r}\leq \sum_{k=1}^3z_k\lambda_k^\eta+c\eta ,\qquad
0 = \Mint v_x \,\mathrm dx\geq \sum_{k=1}^3z_k\lambda_k^\eta-c\eta
\eeq
where $\tilde{r}$ is defined as in \eqref{tilder}, and has been bounded according to the estimate in \eqref{bound for tilder}. By combining \eqref{ineq 001}--\eqref{ineq 002} we are led to
\beq
\label{ineq 003}
 z_{31}^{-1}(1-z_{21}\lambda^\eta_2) + c\eta \geq \lambda_3^\eta\geq z_{31}^{-1}(1-z_{21}\lambda^\eta_2) - c\eta,
\eeq
so that, by \eqref{stimettada2soldi},
\beq
\label{da sotto pure}
\Mint \bigl(\eps^4 v_{xx}^2 +\eps^{-2}W(v_x)+\eps^{-2}v^2\bigr)\,\mathrm dx+ c\eta
\geq
A_0 \Bigl(\lambda_2^\eta+ \frac{K}{z_{31}}(1-z_{21}\lambda^\eta_2)\Bigr).
\eeq
The right hand side of the above inequality is a decreasing function of $\lambda_2^\eta$ so minimised by the biggest admissible $\lambda_2^\eta$. But as $\lambda_3^\eta\geq 0$, \eqref{ineq 003} entails $\lambda_2^\eta\leq z_{21}^{-1} + c\eta$, thus implying
\beq
\label{da sotto super bound}
\Mint \bigl(\eps^4 v_{xx}^2 +\eps^{-2}W(v_x)+\eps^{-2}v^2\bigr)\,\mathrm dx+ c\eta
\geq
A_0 z_{21}^{-1}.
\eeq
Choosing $\eta =  \eps^{\frac1{q+1}}$ and $\eps_1 =\min\{\eta^{q+1}_1,\eps_0\}$, where $\eps_0$ is as in Proposition \ref{upper bnd} we complete the proof of \eqref{tesi inf}. Finally, combining \eqref{da sopra ok} and \eqref{da sotto pure}--\eqref{da sotto super bound} 
we get
$$
c\eps^\xi+A_0z_{21}^{-1}\geq A_0\lambda_2^\eta+z_{31}^{-1}KA_0(1-z_{21}\lambda_2^\eta)\geq A_0z_{21}^{-1}- c\eps^\xi,
$$ 
which is
$$
\bigl|(\lambda_2^\eta-z_{21}^{-1})(1-K\frac{z_{21}}{z_{31}})\bigr|\leq c\eps^\xi.
$$
This implies $|\lambda_2^\eta-z_{21}^{-1}|\leq c\eps^\xi$ which, by \eqref{ineq 003}, concludes the proof.
\end{proof}

\subsection{Proof of Theorem \ref{main Thm 2}}
This follows as a corollary of Theorem \ref{ThmScaling}.\\

By Proposition \ref{compattezza} together with \eqref{tesi inf} we know that every $u_{\eps_j}\in V$ sequence of minimisers for $I^{\eps_j}$, and hence of minimisers for $\E^{\eps_j}$, generates, up to a subsequence, a gradient Young measure $\nu\in \mathrm{GYM}^p(0)$ as $\eps_j\to0$, and that $\supp\nu_x\subset \mathcal{Z}$ almost everywhere in $(0,1)$. As a consequence, $\nu_x$ is of the form
$$
\nu_x = \lambda_1(x)\delta_{z_1}+\lambda_2(x)\delta_{z_2}+\lambda_3(x)\delta_{z_3},
$$
with the $\lambda_i$'s satisfying \eqref{comp 01} for a.e. $x\in(0,1).$ Therefore we just need to show that $\lambda_3=0$ a.e. in $(0,1)$. Let us consider a continuous function $f_3\colon\R\to[0,1]$ which is equal to 1 for all $s\in\R$ such that $|s-z_3|\leq\frac{\eta_0}2$, and $0$ if $|s-z_3|\geq\eta_0$. By arguing as in the proof of Lemma \ref{limite la} we get
$$
0\leq \Mint \lambda_3\,\mathrm dx = \lim_j\Mint f_3(u_{\eps_j,x})\,\mathrm dx \leq \lim_j\bigl(\lambda^\eta_{3}(u_{\eps_j})+c\eps_j^2\eta^{-q}\bigr).
$$
After choosing $\eta = \eps_j^{\frac1{q+1}}$, \eqref{tesi inf 2} gives the sought result.

\section{Some remarks on the assumptions}
\label{remarks on ass}
It is worth spending some words on assumptions (H6)--(H8). Hypothesis (H6) is needed in our construction of a lower bound, but it might be possible to remove it by making the arguments of Section \ref{section 4} more involved. It is easy to check that it fails whenever $z_3>3|z_1|$. On the other hand, as mentioned in the introduction, it turns out that (H7)--(H8) are necessary conditions in order to prove Theorem \ref{main Thm}, and 
the second $\Gamma$--limit would have a different form without these assumptions. Indeed, as explained in the introduction, these hypotheses guarantee that the microstructures used in the construction of Proposition \ref{upper bnd} are energetically preferable to those constructed in the following Propositions, and shown in Figure \ref{fig:fig}. 
\begin{prpstn}
\label{H7 prop}
Assume {\rm(H7)} is not satisfied. Then, there exist $\lambda_1,\lambda_2,\lambda_3\in(0,1)$ satisfying \eqref{comp 01}, $u_\eps\in V$ such that $(u_\eps,\delta_{u_{\eps,x}})\to(0,\nu)$ in $L^2(0,1)\times L_{w^*}^\infty(0,1;\mathcal{M})$, where $\nu_x = \lambda_1\delta_{z_1}+\lambda_2\delta_{z_2}+\lambda_3\delta_{z_3}$ a.e. in $(0,1)$, and
$$\limsup_{\eps\downarrow0} I^\eps(u_\eps)< A_0\lambda_2+B_0\lambda_3.$$ 
\end{prpstn}
\begin{proof}
The proof of the Proposition is very similar to the one of Proposition \ref{upper bnd} in many details. For this reason we skip some long computation and just give the idea of the proof.\\

Let $\hat y\geq 0$ such that the inequality in (H7) does not hold. Let us choose
\beq
\label{lambdas vari}
\lambda_2 = (\hat yz_{31}+z_{21})^{-1}, \qquad \lambda_3 = z_{31}^{-1}-\frac{z_{21}}{z_{31}}\lambda_2,\qquad \lambda_1 = 1-\lambda_2-\lambda_3,
\eeq
so that $\frac{\lambda_3}{\lambda_2}=\hat y$. It is easy to check that $\lambda_1,\lambda_2,\lambda_3\in(0,1)$ and satisfy \eqref{comp 01}. We divide $(0,1)$ in $N_\eps$ subintervals $(x_i,x_{i+1})$ of length $N_\eps^{-1}$, where $x_i=iN_\eps^{-1}$ for $i=0,\dots,N_\eps$. On every interval we construct $v_\eps$ as a suitable continuous approximation (see the proof of Proposition \ref{upper bnd}) of the function (see the derivative of the function in Figure \ref{fig0a})
\[
\hat v_\eps(s)=
\begin{cases}
z_2,		\qquad&\text{if $0 \leq |s-(2N_\eps)^{-1}|\leq \lambda_2 (2N_\eps)^{-1}$},\\
z_3,	\qquad&\text{if $\lambda_2 N_\eps^{-1} < |s-(2N_\eps)^{-1}| \leq  \lambda_2 (2N_\eps)^{-1}+\lambda_3 (2N_\eps)^{-1}$}\\
z_1,		\qquad&\text{if $\lambda_2 (2N_\eps)^{-1}+\lambda_3 (2N_\eps)^{-1}\leq |s-(2N_\eps)^{-1}|\leq (2N_\eps)^{-1}$}.
\end{cases}
\]
{We remark that, as in the proof of Proposition \ref{upper bnd}, $v_\eps$ must satisfy $\int_0^{N_\eps^{-1}} v_\eps(s)\,\mathrm{d}s=0,$ and $v_\eps(0)=v_\eps(N^{-1}_\eps)$.} Therefore, after defining $w_\eps$ as the $N_\eps^{-1}-$periodic extension of $v_\eps$, we construct $u_\eps$ as in \eqref{uuu}. Now, an argument as the one in the proof of Proposition \ref{upper bnd}, allows us to prove that
\[
\begin{split}
I^\eps(u_\eps)\leq 3^\frac23 2^{-\frac23}(E_0+2E_1)^\frac23\Bigl(z_2^2z_{21}\lambda_2^3+z_3^2z_{31}\lambda_3^3+3\lambda_2\lambda_3z_2z_{31}(\lambda_3z_3+\lambda_2z_2)	\Bigr)^\frac13+ c\eps^\xi = \lambda_2f_7 \bigl(\hat{y}\bigr)
+ c\eps^\xi
\end{split}
\]
for some $\xi>0,$ and that \eqref{da combino} holds. Here we have used that, by construction, $\frac{\lambda_3}{\lambda_2}=\hat y$. As $\hat{y}$ contradicts (H7), we have
\beq
\label{nongamma}
I^\eps(u_\eps)< A_0\lambda_2 + B_0\lambda_3 + c\eps^\xi.
\eeq
{Furthermore, by arguing as to get \eqref{da combino}, we have
\beq
\label{da combino 2}
\bigl| \mathscr{L}\bigl( (0,1)\cap\{| u_x-z_2|\leq \sigma\} \bigr) - \lambda_{2}\bigr| + \bigl| \mathscr{L}\bigl( (0,1)\cap\{| u_x-z_3|\leq \sigma\} \bigr) - \lambda_3\bigr|\leq c \eps^\zeta.
\eeq}
Thus, taking the $\limsup$ in \eqref{nongamma}, by Lemma \ref{limite la} we obtain the sought result.
\end{proof}

\begin{figure}
\begin{subfigure}{.33\textwidth}
  \centering
  \includegraphics[width=.99\linewidth]{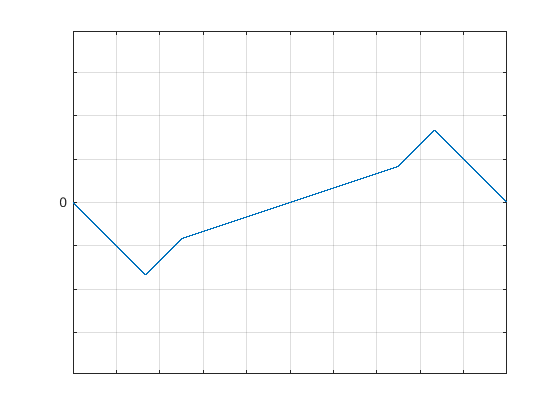}
  \caption {}
  \label{fig0a}
\end{subfigure}%
\begin{subfigure}{.33\textwidth}
  \centering
  \includegraphics[width=.99\linewidth]{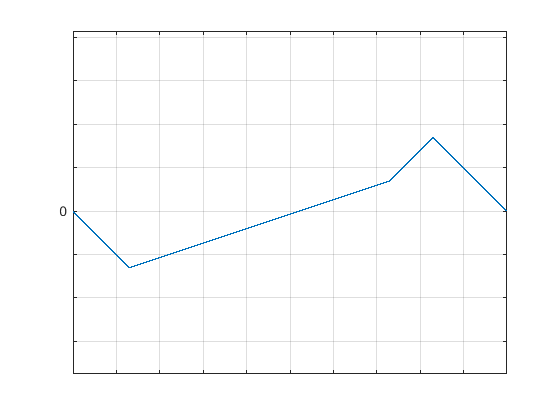}
  \caption{}
  \label{fig0b}
\end{subfigure}
\begin{subfigure}{.33\textwidth}
  \centering
  \includegraphics[width=.99\linewidth]{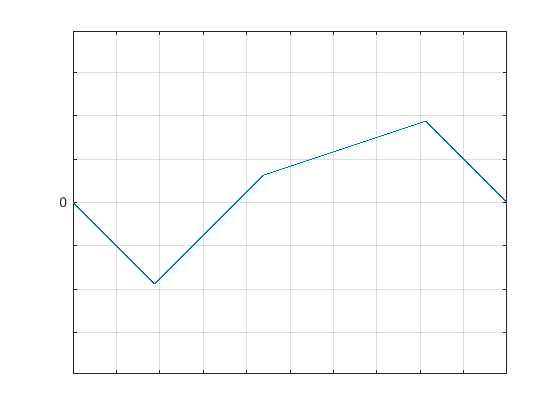}
  \caption{}
  \label{fig0c}
\end{subfigure}
\caption{\ref{fig0a}. Microstructure of lower energy in case (H7) does not hold. Microstructures of low energy in case (H8) does not hold are represented in Figures \ref{fig0b} and \ref{fig0c}, respectively in the case where $\hat y\leq \sqrt{\frac{z_2z_{21}}{z_3z_{31}}}$ and $\hat y> \sqrt{\frac{z_2z_{21}}{z_3z_{31}}}$.}
\label{fig:fig}
\end{figure}

\begin{prpstn}
\label{H8 prop}
Assume {\rm(H8)} is not satisfied. Then, there exist $\lambda_1,\lambda_2,\lambda_3\in(0,1)$ satisfying \eqref{comp 01}, $u_\eps\in V$ such that $(u_\eps,\delta_{u_{\eps,x}})\to(0,\nu)$ in $L^2(0,1)\times L_{w^*}^\infty(0,1;\mathcal{M})$, where $\nu_x = \lambda_1\delta_{z_1}+\lambda_2\delta_{z_2}+\lambda_3\delta_{z_3}$ a.e. in $(0,1)$, and
$$\limsup_{\eps\downarrow0} I^\eps(u_\eps)< A_0\lambda_2+B_0\lambda_3.$$ 
\end{prpstn}
\begin{proof}
Again, the proof of this Proposition is very similar to the one of Proposition \ref{upper bnd} and of Proposition \ref{H7 prop} in many details. For this reason we skip some long computation and just give the idea of the proof.\\

Let $\hat y\geq 0$ such that (H8) does not hold. Let us choose $\lambda_1,\lambda_2,\lambda_3$ that satisfy \eqref{comp 01} as in \eqref{lambdas vari}. Again, we divide $(0,1)$ in $M_\eps$ subintervals $(x_i,x_{i+1})$ of length $M_\eps^{-1}$, where $x_i=iM_\eps^{-1}$ for $i=0,\dots,M_\eps$. We have two cases:
$$
0\leq \hat{y}\leq \sqrt{\frac{z_2z_{21}}{z_3z_{31}}},\qquad\text{ and }\qquad \sqrt{\frac{z_2z_{21}}{z_3z_{31}}}< \hat{y}.
$$
In the first case (see the derivatives of the function in Figure \ref{fig0b}), we define $\omega^a:=(z_2z_{21}- z_3z_{31}\hat y^2)(2z_2(z_{21} + \hat y z_{31}))^{-1}$ and
\[
\hat v_{a}^\eps (s)=
\begin{cases}
z_1,		\qquad&\text{if $0 \leq s< \frac{z_2}{z_1}(1-\omega^a)\lambda_2 M_\eps^{-1}$},\\
z_2,		\qquad&\text{if $\frac{z_2}{z_1}(1-\omega^a)\lambda_2 M_\eps^{-1}\leq s<\Bigl(\frac{z_2}{z_1}(1-\omega^a)+1\Bigr)\lambda_2 M_\eps^{-1}  $},\\
z_3,		\qquad&\text{if $\Bigl(\frac{z_2}{z_1}(1-\omega^a)+1\Bigr)\lambda_2 M_\eps^{-1}\leq s<\Bigl(\frac{z_2}{z_1}(1-\omega^a)+1\Bigr)\lambda_2 M_\eps^{-1}+\lambda_3 M_\eps^{-1}  $},\\
z_1,		\qquad&\text{if $\Bigl(\frac{z_2}{z_1}(1-\omega^a)+1\Bigr)\lambda_2 M_\eps^{-1}+\lambda_3 M_\eps^{-1}\leq s$}.
\end{cases}
\]
In the second (see the derivatives of the function in Figure \ref{fig0c}), we define $\omega^b:=-\omega^a$ and
\[
\hat v_{b}^\eps (s)=
\begin{cases}
z_1,		\qquad&\text{if $0 \leq s< \frac{z_3}{z_1}(1-\omega^b)\lambda_3 M_\eps^{-1}$},\\
z_3,		\qquad&\text{if $\frac{z_3}{z_1}(1-\omega^b)\lambda_3 M_\eps^{-1}\leq s<\Bigl(\frac{z_3}{z_1}(1-\omega^b)+1\Bigr)\lambda_3 M_\eps^{-1}  $},\\
z_2,		\qquad&\text{if $\Bigl(\frac{z_3}{z_1}(1-\omega^b)+1\Bigr)\lambda_3 M_\eps^{-1}\leq s<\Bigl(\frac{z_3}{z_1}(1-\omega^b)+1\Bigr)\lambda_3 M_\eps^{-1}+\lambda_2 M_\eps^{-1}  $},\\
z_1,		\qquad&\text{if $\Bigl(\frac{z_3}{z_1}(1-\omega^b)+1\Bigr)\lambda_3 M_\eps^{-1}+\lambda_2 M_\eps^{-1}\leq s$}.
\end{cases}
\]
Now, let us consider suitable continuous approximations $v_a^\eps$ and $v_b^\eps$ of $\hat v_a^\eps$ and $\hat v_b^\eps$, which can be obtained in the same way as the one in Proposition \ref{upper bnd}. Again, we remark that $v^\eps_l$ for $l=a,b$ must satisfy $v_l^\eps(0)=v_l^\eps(M^{-1}_\eps)$ and $\int_0^{M^{-1}_\eps}v_l^\eps(s)\,\mathrm{d}s=0.$ Let $w_a^\eps, w_b^\eps$ be the $M_\eps^{-1}-$periodic extensions of $v_a^\eps$ and $v_b^\eps$ respectively, and define $u_\eps$ as in \eqref{uuu}. An argument as the one in Proposition \ref{upper bnd} allows hence to prove
\[
\begin{split}
I^\eps(u_\eps)\leq 3^\frac23 (E_0+E_1)^\frac23\Bigl(z_2^2z_{21}\lambda_2^3+z_3^2z_{31}\lambda_3^3- 3\lambda_2^3f_0(\hat{y})	\Bigr)^\frac13+ c\eps^\xi
=\lambda_2 f^\frac13_8(\hat{y})+ c\eps^\xi,
\end{split}
\]
for some $\xi>0$, for some $\eps_0>0$, and for every $\eps\in(0,\eps_0)$. Here $f_0$ is as in \eqref{deff0}, and we used the fact that, by construction, $\frac{\lambda_3}{\lambda_2}=\hat y$. The fact that $\hat y$ violates the inequality in (H8) yields
\beq
\label{mobasta}
I^\eps(u_\eps) < A_0\lambda_2 + B_0\lambda_3+ c\eps^\xi.
\eeq
Furthermore, by arguing as in the proof of Proposition \ref{upper bnd} we can prove estimates as the ones in \eqref{da combino 2}. Therefore, by taking the $\limsup$ in \eqref{mobasta} and exploiting Lemma \ref{limite la} we conclude the proof of the proposition.
\end{proof}

\subsection{Two examples}
\label{section 7}
An easy example where hypotheses (H1)-(H8) hold is when $$W(s) = (s-1)^2(s+1)^2(s-3^{-1})^2.$$
Indeed, in this case $E_0\approx 1.054$, $E_1\approx0.165$, $A_0\approx0.718$, $B_0\approx1.883$, $z_1 = -1,$ $z_2 = \frac13$, $z_3=1$. It is trivial to check that in this context (H1)--(H5) hold. Hypotheses (H6)--(H8) are here verified graphically (cf. Figure \ref{fig:fig2}). It can be proved that (H6)--(H8) hold for $W$ of the form
$$
W(s) = (s-1)^2(s+1)^2(s-z_2)^2,
$$
whenever $z_2\in(-0.49,0.49)$. The bound on $z_2$ is not sharp.\\

\begin{figure}
\begin{subfigure}{.33\textwidth}
  \centering
  \includegraphics[width=.9\linewidth]{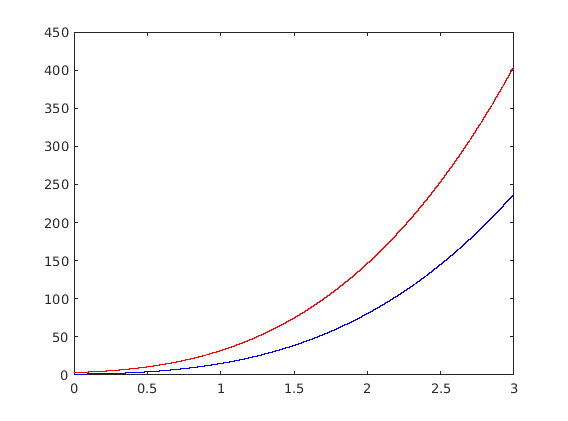}
  \caption{}
  \label{fig99}
\end{subfigure}%
\begin{subfigure}{.33\textwidth}
  \centering
  \includegraphics[width=.9\linewidth]{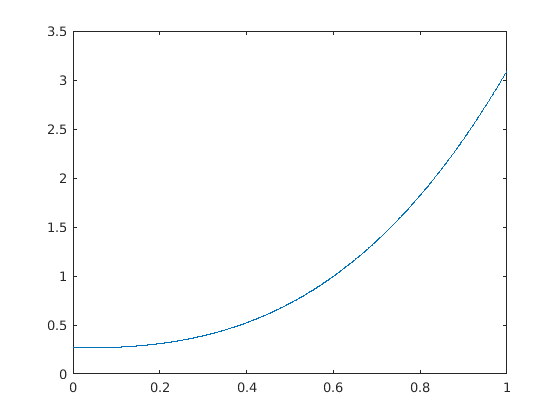}
  \caption{}
  \label{fig1}
\end{subfigure}%
\begin{subfigure}{.33\textwidth}
  \centering
  \includegraphics[width=.9\linewidth]{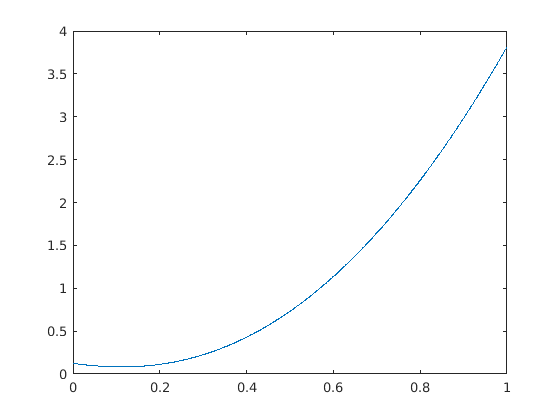}
  \caption{}
  \label{fig2}
\end{subfigure}
\caption{Verification of the hypotheses (H6)--(H8) for the examples of Section \ref{section 7}. Figure \ref{fig99} is the graphical verification of (H6) for the two examples: in blue the case $z_2=\frac13$, in red the one with $z_2=\frac12$. Figure \ref{fig1} and \ref{fig2} verify (H7) and (H8) in the example with $z_2=\frac13.$}
\label{fig:fig2}
\end{figure}

On the other hand, let us consider $$W(s) = (s-1)^2(s+1)^2(s-2^{-1})^2,$$
where, in our notation, $E_0\approx 1.406$, $E_1\approx0.073$, $A_0\approx1.186$, $B_0\approx2.143$, $z_1 = -1,$ $z_2 = \frac12$, $z_3=1$. In this case (H1)--(H6) hold (cf. Figure \ref{fig99}). However, hypotheses (H7) and (H8) fail respectively in a neighbourhood of $\hat y_7 = 0.585$ and $\hat y_8 = 0.204$. Here, it is energetically very cheap to pass from $z_2$ to $z_3$ so other microstructures are energetically favourable for $\frac{\lambda_3}{\lambda_2}$ close to $\hat{y}_7$ or $\hat{y}_8$. Nonetheless, as $z_3\leq 3|z_1|$, thanks to Theorem \ref{main Thm 2} we can still select minimizing gradient Young measures for $\EE^0$ by means of vanishing interfacial energy.\\

\textbf{Acknowledgements:} This work was supported by the Engineering and Physical Sciences Research Council [EP/L015811/1]. The author would like to acknowledge the anonymous reviewers for improving this paper with their comments, and providing the current version of Lemma \ref{dasottointegrale}. The author would also like to thank John Ball for the useful suggestions and discussions.
\footnotesize
\bibliographystyle{plain}
\bibliography{biblio}

\end{document}